\documentclass[10pt, english,oneside,reqno]{amsart}

\usepackage[utf8]{inputenc}
\usepackage[english]{babel}
\usepackage{hyperref}
\usepackage{color}
\usepackage{enumerate}
\usepackage{fancyhdr}

\usepackage{amsmath}
\usepackage{amscd}
\usepackage{amsfonts}
\usepackage{amsthm}
\usepackage{amstext}
\usepackage{amssymb}
\usepackage{amsbsy}
\usepackage{mathrsfs}
\usepackage{tikz-cd}
\usepackage{mathtools}
\usepackage{enumitem}

\usepackage{float}
\usepackage{caption}
\usepackage{subcaption}
\usepackage{graphicx}

\usepackage[alphabetic,abbrev]{amsrefs}

\usepackage{url}

\theoremstyle{plain}
\newtheorem{theorem}{Theorem}[section]

\newtheorem{definition}[theorem]{Definition}
\newtheorem{proposition}[theorem]{Proposition}
\newtheorem{corollary}[theorem]{Corollary}
\newtheorem{lemma}[theorem]{Lemma}

\theoremstyle{remark}

\numberwithin{equation}{section}
\numberwithin{figure}{section}

\DeclareMathOperator{\var}{Var}

\newcommand{\eps}{\varepsilon}
\newcommand{\tr}{\mathrm{tr}}

\newcommand{\Z}{\mathbb{Z}}

\newcommand{\R}{\mathbb{R}}

\newcommand{\E}{\mathbb{E}}
\newcommand{\Var}{\mathrm{Var}}

\newcommand{\Haarof}[1]{m_{#1}}

\begin{document}
	\title[Dimension of contracting on average Self-Similar Measures]{Dimension of contracting on average Self-Similar Measures}
	\author[sk]{Samuel Kittle}
    \author[ck]{Constantin Kogler}

    \email{s.kittle@ucl.ac.uk, kogler@maths.ox.ac.uk}

    \address{Samuel Kittle, Department of Mathematics, University College London, 25 Gordon Street, London WC1H 0AY, United Kingdom}

    \address{Constantin Kogler, Mathematical Institute, University of Oxford, Radcliffe Observatory Quarter, Woodstock Road, Oxford OX2 6GG, United Kingdom}

    \begin{abstract}
        We generalise Hochman's theorem on the dimension of self-similar measures to contracting on average measures and show that a weaker condition than exponential separation on all scales is sufficient. Our proof uses a technique we call the variance summation method, avoiding the use of inverse theorems for entropy. 
    \end{abstract}

    \maketitle

    \tableofcontents
	
    \section{Introduction}

    A central problem in the theory of self-similar measures is to determine their dimension. Significant progress was achieved by Hochman \cite{Hochman2014}, \cite{Hochman2017}, demonstrating that assuming exponential separation on infinitely many scales, the dimension of a self-similar measure can be expressed in terms of the random walk entropy and the Lyapunov exponent. This paper aims to extend Hochman's results to contracting on average measures as well as to show that a weaker condition than exponential separation on all scales is sufficient.  Our approach employs a technique we call variance summation, replacing the use of inverse theorems for entropy. The variance summation method was initially introduced by the first-named author in \cite{Kittle2023} to construct explicit examples of absolutely continuous Furstenberg measures of $\mathrm{SL}_2(\R)$ and was further developed by the authors in the context of self-similar measures on $\R^d$ in \cite{KittleKogler}. 

    Denote by $G = \mathrm{Sim}(\R^d)$ the group of similarities on $\R^d$ and let $O(d)$ be the group of orthogonal $d \times d$ matrices. For each $g \in G$ there exists a scalar $\rho(g) > 0$, an orthogonal matrix $U(g) \in O(d)$ and a vector $b(g) \in \R^d$ such that $g(x) = \rho(g)U(g)x + b(x)$ for all $x \in \R^d$. A similarity is called contracting if $\rho(g) < 1$ and expanding when $\rho(g) > 1$.

    Given a probability measure $\mu$ on $G$ we define the Lyapunov exponent as $$\chi_{\mu} = \mathbb{E}_{g \sim \nu}[\log \rho(g)]$$ whenever it exists. Throughout this paper we use the following terminology.

    \begin{definition}
        If $\chi_{\mu} < 0$, we call $\mu$ \textbf{contracting on average}. Moreover, if every $g\in \mathrm{supp}(\mu)$ is contracting, we say that $\mu$ is \textbf{contracting}.
    \end{definition}

    It is well-known (\cite{Hutchinson1981}, \cite{BarnsleyElton1988}) that when $\mu$ is a finitely supported measure on $G = \mathrm{Sim}(\R^d)$ with $\chi_{\mu} < 0$, then there exists a unique probability measure $\nu$ on $\R^d$ satisfying $\mu * \nu = \nu$. The measure $\nu$ is called the self-similar measure of $\mu$. We say that $\nu$ is a contracting on average self-similar measure whenever $\mu$ is. 

    To introduce further notation, for a finitely supported probability measure $\mu$ on $G$ denote the random walk entropy as $$h_{\mu} = \lim_{n \to \infty} \frac{1}{n}H(\mu^{*n}) = \inf_{n \geq 1} \frac{1}{n}H(\mu^{*n}),$$ where $H$ is the Shannon entropy. Consider on $G$ the metric $$d(g,h) = |\log \rho(g) - \log \rho(h)| + ||U(g) - U(h)|| + |b(g) - b(h)|$$ for $g,h \in G$, $|\circ|$ the euclidean norm and $||\circ||$ the operator norm and define $$\Delta_n = \min\{ d(g,h) \,:\, g,h \in \mathrm{supp}(\mu^{*n}) \text{ with } g\neq h \}$$ and $$M_n = \min\left\{ d(g,h) \,:\, g,h \in \bigcup_{i = 0}^n \mathrm{supp}(\mu^{*i}) \text{ with } g\neq h \right\}.$$ 
    
    Denote by $U(\mu)$ the pushforward of $\mu$ under the map $g \mapsto U(g)$. We call $\mu$  irreducible if the support of $U(\mu)$ acts irreducibly on $\R^d$, meaning that there are no subspaces of $\R^d$ invariant under all elements in $\mathrm{supp}(U(\mu))$ except for the trivial ones $\{0\}$ and $\R^d$. We furthermore say that $\mu$ is without a common fixed point if the similarities in the support of $\mu$ do not have a common fixed point, as otherwise the Dirac measure at a common fixed point is the self-similar measure. 

    It is well-established (cf. for example \cite{Feng2023}) that when $\mu$ is contracting on average, then $\nu$ is exact dimensional, that is there is $\alpha \in [0,d]$ such that for $\nu$-almost every $x \in \R^d$ we have $\nu(B_r(x)) = r^{\alpha + o_{\mu,x}(1)}$ as $r\to 0$. The number $\alpha$ is called the dimension of $\nu$ and denoted as $\dim \nu$. It furthermore holds that $$\dim \nu \leq  \min\left\{ d, \frac{h_{\mu}}{|\chi_{\mu}|} \right\}.$$ 

    It is conjectured that whenever $\mu$ is a finitely supported, contracting on average and irreducible probability measure on $G$ without a common fixed point, then $\dim \nu  = \min\{ d, \frac{h_{\mu}}{|\chi_{\mu}|} \}.$ This was proved by Hochman for $d = 1$ in \cite{Hochman2014} and for arbitrary $d$ in \cite{Hochman2017} under the additional assumptions that $\mu$ is contracting, the elements in the support of $\mu$ generate a free semi-group and that for some $c > 0$ it holds that $\Delta_n \geq e^{-cn}$ for infinitely many $n \geq 1$. In this paper we generalise Hochman's result to contracting on average measures and we do not assume that the elements in the support of $\mu$ generate a free semi-group. We work with $M_n$ instead of $\Delta_n$ in order to apply previous results by the authors on the entropy of stopped random walks (\cite{KittleKoglerEntropy}). 

    \begin{theorem}(Generalisation of Hochman's theorem)\label{GeneralHochman}
        Let $\mu$ be a finitely supported, contracting on average and irreducible probability measure on $G$ without a common fixed point. Furthermore assume that there is $c > 0$ such that $M_n \geq e^{-cn}$ for infinitely many $n \geq 1$. Then \begin{equation}\label{HochmanDimensionFormula}
            \dim \nu = \min\left\{ d, \frac{h_{\mu}}{|\chi_{\mu}|} \right\}.
        \end{equation}
    \end{theorem}

    We note that Hochman \cite{Hochman2017} didn't assume that $\mu$ is irreducible, yet only the weaker condition that on a non-trivial $U(\mu)$-invariant subspace $V\subset \R^d$ it holds for $\nu$ almost every $x\in \R^d$ that the conditional measure $\nu_{V + x}$ on $V +  x$ satisfies $\dim \nu_{V + x} = \dim V$. For simplicity it is assumed in this paper that $\mu$ is irreducible. 
    
    Our method also allows us to prove the following weakening of the separation condition that $M_n \geq e^{-cn}$ for all $n\geq 1$. The reader may observe that in Theorem~\ref{MainWeakHochman} we require information on all scales, which contrasts  Theorem~\ref{GeneralHochman} where  information on only infinitely many scales is needed. 

    \begin{theorem}(Weakening of separation condition on all scales)\label{MainWeakHochman}
         Let $\mu$ be a finitely supported, contracting on average and irreducible probability measure on $G$ without a common fixed point. Furthermore assume that there is $\varepsilon > 0$ such that for all sufficiently large $n \geq 1$,
         \begin{equation}\label{WeakerSeparation}
             \log M_n \geq - n\exp((\log n)^{1/3-\varepsilon}).
         \end{equation}
         Then \eqref{HochmanDimensionFormula} holds. 
    \end{theorem}

    We note that condition \eqref{WeakerSeparation} is weaker than assuming $M_n \geq e^{-n \cdot p(\log n)}$ for any real polynomial $p$, yet stronger than assuming $M_n \geq e^{-n^{1 + \eps}}$ for any $\eps > 0$. It is well-known that when $\mu$ is supported on similarities with algebraic coefficients then for some $c > 0$ it holds that $\Delta_n \geq M_n \geq e^{-cn}$ for all $n \geq 1$. On the other hand, in \cite{Hochman2017} it is shown that the latter holds generically for parametrized families of contracting self-similar measures.

    Denote by $\nu_{\lambda}$ the Bernoulli convolution of parameter $\lambda$. In recent landmark work, Varjú \cite{Varju2019b} proved that $\dim \nu_{\lambda} = 1$ for all transcendental $\lambda \in (1/2,1)$. Varjú's proof relies on subtle approximation results of $\lambda$ by algebraic numbers and on a result similar yet weaker to Theorem~\ref{MainWeakHochman}, which was contained in Breuillard-Varjú \cite{BreuillardVarju2019}. Roughly speaking, Breuillard-Varjú proved in \cite{BreuillardVarju2019} for the Bernoulli convolution $\nu_{\lambda}$, as an important proof step of their main result, that if $\log M_n \geq - Cn \log n$ for some $C > 0$ and all $n\geq 1$ then \eqref{HochmanDimensionFormula} holds. The techniques presented in this paper are more flexible than the ones used in the latter proof step and therefore lead to the strengthening \eqref{WeakerSeparation} for especially also inhomogeneous self-similar measures and in arbitrary dimensions.

    We proceed with an outline of proofs in section~\ref{Outline} followed by setting up relevant notation in section~\ref{Notation}. We discuss we few preliminaries in section~\ref{PreliminariesSection}. Entropy bounds are established in section~\ref{SectionEntropy} and we perform the variance summation method in section~\ref{VarianceSummation}. The proof of Theorem~\ref{GeneralHochman} and of Theorem~\ref{MainWeakHochman} is concluded in section~\ref{SectionConclusion}. 

    \subsection*{Acknowledgment} The first-named author gratefully acknowledges support from the Heilbronn Institute for Mathematical Research. This work is conducted during the second-named author's PhD studies at the University of Oxford.

    \section{Outline and Notation}
    \label{OutlineNotation}

\subsection{Outline of proofs} \label{Outline} 

We give an outline of the proof of Theorem~\ref{GeneralHochman} and Theorem~\ref{MainWeakHochman}. We use techniques from \cite{KittleKogler} and the entropy bounds from \cite{KittleKoglerEntropy}. Hochman's proof \cite{Hochman2017} relies on inverse theorems for entropy. We bypass the latter by using a decomposition theory for stopped random walks and by summing the amount of variance we gain at each scale. Indeed, to introduce notation let $\gamma_1, \gamma_2, \ldots$ be independent $\mu$-distributed random variables. For a stopping time $\tau$ write $q_{\tau} = \gamma_1 \cdots \gamma_{\tau}$. Note that if $x$ is a sample of $\nu$ independent from $\gamma_1,\gamma_2,\ldots$ then $q_{\tau} x$ is also a sample $\nu$. The basic idea of our proof is to decompose $q_{\tau} x$ as a sum 
\begin{equation}\label{BasicDecomposition}
q_{\tau} x \approx x_0 + X_1 + \dots + X_{n}
\end{equation}
of random variables and to construct a $\sigma$-algebra $\mathscr{A}$ such that $x_0$ is $\mathscr{A}$-measurable and the $X_1, \ldots , X_n$ are conditionally independent given $\mathscr{A}$. We also require the $X_i$ to have some variance after conditioning on $\mathscr{A}$.

In order to prove that $\dim \nu = \min\{ d, \frac{h_{\mu}}{|\chi_{\mu}|}\}$, our strategy of proof assumes that $\dim \nu < \frac{h_{\mu}}{|\chi_{\mu}|}$
from which we will deduce that $\dim\nu = d$.  To conclude the latter, we aim to show that for every $C > 0$ and each sufficiently small scale $r > 0$ and a suitable stopping time $\tau$ we can find a decomposition \eqref{BasicDecomposition} such that for all $i \in [n]$ it holds that 
\begin{equation}
|X_i| \leq C^{-1}r 
\end{equation} and with probability $1 - C^{-1}$ we have \begin{equation}\label{Outline:DecompositionGoal}
\sum_{j = 1}^n \mathrm{Var} (X_j | \mathscr{A}) \geq Cr^2I, 
\end{equation}  where $\Var (X_j | \mathscr{A})$ is the covariance matrix of $X_j$ conditional on $\mathscr{A}$ and we denote by $\geq$ the partial order defined in \eqref{MatrixPartialOrder}. Note that $\Var(X_j | \mathscr{A})$ is an $\mathscr{A}$-measurable random variable.

We may conclude from a Berry-Essen type estimate that if \eqref{Outline:DecompositionGoal} holds for arbitrarily large $C$, then $\nu$ has dimension $d$. Indeed, it will follow that if $\Sigma = \sum_{j = 1}^n \mathrm{Var}(X_j | \mathscr{A})$, then roughly speaking $$\mathcal{W}_1(q_{\tau} x | \mathscr{A}, x_0 + \mathcal{N}(0,\Sigma)) \ll_d C^{-1}r,$$ where $\mathcal{N}(0,\Sigma)$ is the multivariate Gaussian with mean $0$ and variance $\Sigma$ and $\mathcal{W}_1$ is the Wasserstein $L^1$-distance.  As $C \to \infty$, $\nu$ is well-approximated by a smoothened random variable. From the latter it will be straightforward to show that $\dim \nu = d$ (see Proposition~\ref{propositon:normal_to_entropy}).

\subsubsection*{\normalfont\textbf{From Decomposition on $\R^d$ to Decomposition on $G$}} Instead of constructing a decomposition \eqref{Outline:DecompositionGoal} on $\R^d$, we will decompose $q_{\tau}$ on $G$ into 
\begin{equation}\label{Outline:Decomposition}
q_{\tau} = g_1\exp(U_1) g_2\exp(U_2) \cdots g_n \exp(U_n)
\end{equation}
for random variables $g_1, \ldots , g_n$ on $G$ and $U_1, \ldots , U_n$ on the Lie algebra $\mathfrak{g}$ of $G$ and $\exp: \mathfrak{g} \to G$ the exponential map. We will construct such a decomposition with $g_1, \ldots , g_n$ being $\mathscr{A}$-measurable and $U_1, \ldots , U_n$ being conditionally independent given $\mathscr{A}$. In order to express $q_{\tau}x$ as a sum of random variables using \eqref{Outline:Decomposition}, we apply Taylor's theorem in Proposition~\ref{MainTaylorBound} to deduce 
\begin{equation}\label{Outline:TaylorBound}
q_{\tau}x \approx g_1\cdots g_n x + \sum_{i = 1}^n \zeta_i(U_i),
\end{equation}
where $$\zeta_i = D_u(g_1g_2\cdots g_i \exp(u)g_{i + 1}g_{i + 2}\cdots g_n x)|_{u = 0}.$$

For notational convenience denote $g_i' = g_1\cdots g_i$. It can then be shown that under suitable assumptions on the $g_i$, it holds that 
\begin{equation}\label{Outline:VarTraceBound}
\var(\zeta_i(U_i)|\mathscr{A}) \geq  c_1 \rho(g_i')^2\tr(U_i|\mathscr{A}) = c_1 \tr(\rho(g_i')U_i|\mathscr{A})I
\end{equation}
for some constant $c_1 > 0$ depending only on $\mu$ and where we denote by $\tr(U_i|\mathscr{A})$ the trace of the covariance of $(U_i|\mathscr{A})$. So in order to achieve \eqref{Outline:DecompositionGoal}, we require that 
\begin{equation}\label{Outline:DecompositionNewGoal}
|U_i| \leq \rho(g_i')^{-1}r \quad\quad \text{ and } \quad\quad \sum_{i = 1}^n \tr(\rho(g_i')U_i|\mathscr{A}) \geq C^3 c_1^{-1} r^2
\end{equation}
for the constant $C$ from \eqref{Outline:DecompositionGoal} and the second bound holding with probability $1 - C^{-1}$. Note that to deduce \eqref{Outline:DecompositionGoal} from \eqref{Outline:DecompositionNewGoal} we replace $U_i$ by $C^{-1}U_i$ and use \eqref{Outline:VarTraceBound}.

\subsubsection*{\normalfont\textbf{Entropy Gap and Trace Bounds for Stopped Random Walk}}

To show \eqref{Outline:DecompositionNewGoal} one first establishes suitable entropy gap results to deduce trace bounds at various scales. Indeed, denote for some $a \geq 1$ by $H_a(q_{\tau}; r_1 | r_2)$ the entropy between scales $r_1 < r_2$ as defined in \eqref{DefEntScales}, which measures how much more information $q_{\tau}$ has on scale $r_1$ than on scale $r_2$.  

We refer to \cite{KittleKoglerEntropy} for a discussion on the entropy between scales. Write for $\kappa > 0$ $$\tau_{\kappa} = \inf\{ n \geq 0 \,:\, \rho(q_n) \leq \kappa  \}.$$ Then it will follow from \cite{KittleKoglerEntropy}*{Theorem 1.2} that under the assumptions of Theorem~\ref{GeneralHochman} and that $\dim \nu < \frac{h_{\mu}}{|\chi_{\mu}|}$ it holds that for infinitely many scales $\kappa_n$ with $\kappa_n \downarrow 0$ as $n \to \infty$ that  
\begin{equation}\label{Outline:FirstEntropyEstimate}
H_a(q_{\kappa_n}; \kappa_n^{\beta}| \kappa_n^{\delta}) \geq \alpha \log \kappa_n^{-1}
\end{equation}
for suitable constants $\alpha, \beta$ and $\delta$ depending on $\mu$.

To convert \eqref{Outline:FirstEntropyEstimate} to a trace bound we use the following notation: For a random variable $g$ and a scale $r > 0$ we denote  by $$\mathrm{tr}(g; r)$$ the supremum of all the values $t \geq 0$ such that we can express $$g = h\exp(U)$$ for some $\sigma$-algebra $\mathscr{A}$, some $\mathscr{A}$-measurable $G$-valued random variable $h$  and such that $U$ is a $\mathfrak{g}$-valued random variable satisfying \begin{equation}\label{TraceScaleDef}
|U| \leq r \quad \text{ and } \quad \E[\tr(U|\mathscr{A})] \geq tr^2,
\end{equation} where again $\tr(U|\mathscr{A})$ is the trace of the covariance matrix of $(U|\mathscr{A})$. Up to for the purposes of this outline negligible error terms, it is shown in \cite{KittleKoglerEntropy}*{Theorem 1.4} that 
\begin{equation}\label{EntropyTraceComparison}
\tr(g; 2ar) \gg_a H_a(g; r|2r).
\end{equation}

Since $H_a(g; r|2^{\ell}r) = \sum_{i = 0}^{\ell-1} H_a(g;2^{i}r| 2^{i + 1}r )$, by a telescoping sum argument (\cite{KittleKoglerEntropy}*{Proposition 1.5}) we can deduce from \eqref{Outline:FirstEntropyEstimate} and \eqref{EntropyTraceComparison} that there exists  infinitely many scales $\kappa_n$ with $\kappa_n \downarrow 0$ such that for some $r_n \in (\kappa_n^{\beta}, 2a\kappa_n^{\delta})$ we have for an altered constant $\alpha$,
\begin{equation}\label{OutlineFirstTraceBound}
\tr(q_{\tau_{\kappa_n}}; r_n) \geq \alpha.
\end{equation} The final part of the proof will be to sum up all the contributions at all scales $r_n$, which we will outline below. 

In order to prove Theorem~\ref{MainWeakHochman}, we will require sharper entropy and trace bounds. Indeed, up to negligible errors we will be able to show in Proposition~\ref{proposition:trace_sum_second} that if  for some $B > 0$ we have 
\begin{equation}\label{OutlineTraceSumSecondAssumption}
\log M_n \geq -n\exp((\log n)^B)
\end{equation} then for $\kappa$ sufficiently small there is collection of sufficiently separated scales $s_1, \ldots , s_m$ such that $$s_i \in (\kappa^{\exp((\log \log \kappa^{-1})^{B})}, 2(\log \log \kappa^{-1})^{\frac{B}{2}} \kappa^{\delta}),$$
as well as 
\begin{equation}\label{OutlineTraceSumSecond}
\sum_{i = 1}^m \mathrm{tr}(q_{\tau_{\kappa}}; s_i) \geq \frac{\alpha}{(\log \log \kappa^{-1})^{B}}.
\end{equation} for $m \asymp \exp((\log \log \kappa^{-1})^B)$.

\subsubsection*{\normalfont\textbf{Variance Summation}}

We will explain how to sum up the contributions at various scales in \eqref{OutlineFirstTraceBound} and \eqref{OutlineTraceSumSecond} to conclude that $\dim \nu = d$. As is explained in more detail in section~\ref{VarianceSummation}, we denote for $n,K \in \Z_{\geq 0}$ and $A,r > 0$ with $r \in (0,1)$ by 
\begin{equation}\label{Outline:VNotation}
V(\mu,n,K,\kappa,A; r)
\end{equation}
the maximal amount of variance obtained in the form \eqref{Outline:DecompositionNewGoal} at scale $r$ contained in a decomposition \eqref{Outline:Decomposition} satisfying in essence the following properties:
\begin{enumerate}[label = (\roman*)]
	\item Each of the terms $g_i\exp(U_i)$ is a product of at least $K$ copies of $\gamma_i$.
	\item $\rho(\gamma_1 \cdots \gamma_n) \geq \kappa$
\end{enumerate} The parameter $A$ is there to ensure that the Taylor expansion \eqref{Outline:TaylorBound} has a controllable error term (see section~\ref{VarianceSummation}). The most important property for our purposes is that the variance contributions on different scales are additive. Indeed, as stated in Proposition~\ref{VarSumAdds}, if $\{ \rho(g) \,:\, g \in \mathrm{supp}(\mu) \} \subset [R^{-1},R]$ for some $R > 1$ then for all suitable parameters and with $M \geq R$ we have that
\begin{align}\label{OutlineVarSumAdds}
\MoveEqLeft V(\mu,  n_1 + n_2, K, R^{-1} M^{-1} \kappa_1 \kappa_2, A; r) \nonumber \\&\geq V(\mu, n_1, K, \kappa_1, A; r) + V(\mu,  n_2, K, \kappa_2, A; M \kappa_1^{-1} r).
\end{align}  The variable choice of $M$ allows us to sum up information from  different scales, provided they are separated by at least $R\kappa_1^{-1}$. 

Let us assume now that we have a sequence $V_i = V(\mu,n_i,K_i,\kappa,A; r_i)$ of decompositions \eqref{Outline:VNotation}. Then we will show in section~\ref{SectionConclusion} that if, as $i \to \infty$, it holds that \begin{equation}
V_i / \log n_i \to \infty, \quad \quad \quad \quad  K_i / \log n_i \to \infty\label{outline_eq:v_i_k_i_diverge}
\end{equation}
and
\begin{equation}
\frac{\log r_i^{-1} - \log \kappa_i^{-1}}{n_i} \to \infty \label{outline_eq:taylor_term_dies}
\end{equation} then $\dim \nu = d$. It will become apparent in section~\ref{SectionConclusion} how these assumptions arise. 

It will be straightforward to construct from \eqref{OutlineFirstTraceBound} and \eqref{OutlineVarSumAdds} sequences such that \eqref{outline_eq:v_i_k_i_diverge} and \eqref{outline_eq:taylor_term_dies} hold, which concludes the proof of Theorem~\ref{GeneralHochman}. For Theorem~\ref{MainWeakHochman} we will combine the more subtle bounds \eqref{OutlineTraceSumSecondAssumption} at various scales to conclude the claim (see section~\ref{SectionConclusion}). 

\subsection{Notation} \label{Notation} The same notation as in \cite{KittleKogler} and \cite{KittleKoglerEntropy} is used. We write the asymptotic notation $A \ll B$ or $A = O(B)$ to denote that $|A| \leq CB$ for a constant $C > 0$. If the constant $C$ depends on additional parameters we add subscripts. Moreover, $A \asymp B$ denotes $A \ll B$ and $B \ll A$. 

For an integer $n \geq 1$ we abbreviate $[n] = \{ 1,2,\ldots , n \}$.

Given two positive semi-definite symmetric real $d \times d$ matrices $M_1$ and $M_2$ we write 
\begin{equation}\label{MatrixPartialOrder}
M_1 \geq M_2 \quad\quad \text{ if and only if } \quad\quad x^TM_1 x \geq x^TM_2 x \quad \text{ for all } x \in \R^d.
\end{equation}

For a random variable $X$ on $\R^d$ we denote by $\mathrm{Var}(X)$ the covariance matrix of $X$ and by $\mathrm{tr}(X) = \tr\, \mathrm{Var}(X)$ the trace of the covariance matrix.

Given a metric space $(M,d)$, $p \in [1,\infty)$ and two probability measures $\lambda_1$ and $\lambda_2$ on $M$, we define 
\begin{equation}\label{WassersteinDef}
\mathcal{W}_p(\lambda_1, \lambda_2) = \inf_{\gamma \in \Gamma(\lambda_1,\lambda_2)} \left( \int_{M\times M} d(x,y)^p \, d\gamma(x,y)\right)^{\frac{1}{p}},
\end{equation}
where $\Gamma(\lambda_1,\lambda_2)$ is the set of couplings of $\lambda_1$ and $\lambda_2$, i.e. of probability measures $\gamma$ on $M \times M$ whose projections to the first coordinate is $\lambda_1$ and to the second is $\lambda_2$.

Throughout this paper we fix $d \geq 1$ and write $G = \mathrm{Sim}(\R^d)$. The Lie algebra of $G$ will be denoted $\mathfrak{g}$ and $\ell = \dim \mathfrak{g}$. For $x \in \R^d$ consider the map $$ w_x: \mathfrak{g} \to \R^d, \quad\quad u \mapsto \exp(u)x.$$ Denote by 
\begin{equation}\label{psidef}
\psi_x = D_0w_x : \mathfrak{g} \to \R^d
\end{equation}
the differential at zero of $w_x$. 

Note that we can embed $G = \mathrm{Sim}(\R^d)$ into $\mathrm{GL}_{d + 1}(\R)$ via the map $$g \mapsto \begin{pmatrix}
r(g)U(g) & b(g) \\ 0 & 1
\end{pmatrix}.$$  Therefore we can write $u\in \mathfrak{g}$ as $u = (\begin{smallmatrix}
\alpha & \beta \\ 0 & 0
\end{smallmatrix})$ with $\alpha\in (\R \cdot \mathrm{Id}_d) \oplus \mathfrak{so}_d(\R)$ and $\beta \in \R^d$. Thus it follows that $\psi_x(u) = u(\begin{smallmatrix}
x \\ 1
\end{smallmatrix}) = \alpha x + \beta$. With this viewpoint we also use the following notation
\begin{equation}\label{NotationConvention}
ux = \psi_x(u) = \alpha x + \beta
\end{equation}

We usually consider a fixed probability measure $\mu$ on $G$ and independent samples $\gamma_1, \gamma_2, \ldots$ of $\mu$.  We write for $\kappa > 0$ $$q_n = \gamma_1\cdots \gamma_n \quad\quad  \text{ and } \quad\quad  \tau_{\kappa} = \inf\{ n \geq 1 \,;\, \rho(\gamma_n) \leq \kappa \}.$$

When $\mu$ is a probability measure on $G = \mathrm{Sim}(\R^d)$ and $\nu$ is a probability measure $\R^d$ we denote by $\mu * \nu$ the probability measure uniquely characterized by $$(\mu*\nu)(f) = \int\int f(gx) \, d\mu(g)d\nu(x)$$ for $f\in C_c(\R^d)$. When $\mu = \sum_{i}p_i \delta_{g_i}$ is finitely supported, then 
\begin{equation}\label{Equation:ConcreteConvolution}
\mu * \nu = \sum_{i} p_i g_i\nu,
\end{equation}
where $g_i\nu$ is the pushforward of $\nu$ by $g_i$ defined by $(g_i\nu)(B) = \nu(g_i^{-1}B)$ for all Borel sets $B \subset \R^d$.

Given a random variable $g$ on $G$  by $H(g)$ the Shannon entropy when $g$ is discrete and the differential entropy (with respect to a fixed Haar measure on $G$) when $g$ is absolutely continuous. 

We next define the entropy at a scale and between scales as in \cite{KittleKoglerEntropy}. To do so, we construct a suitable family of smoothing functions. Indeed for given $r > 0$ and $a \geq 1$, denote by $\eta_{a,r}$ a random variable on $\mathfrak{g}$ with density function $f_{a,r}: \mathfrak{g} \to \R$ given by $$f_{a,r}(x) = \begin{cases}
C_{a,r}e^{-\frac{|x|^2}{2r^2}} & \text{if } |x| \leq ar, \\
0 &\text{otherwise},
\end{cases}$$ where $C_{a,r}$ is a normalizing constant to ensure that $f_{a,r}$ integrates to $1$. We furthermore define $$s_{a,r} = \exp(\eta_{a,r}).$$ We then define the entropy at scale $r$ as $$H_a(g;r) = H(g; s_{a,r}) = H(gs_{a,r}) - H(s_{a,r})$$ and the entropy between scales $r_1, r_2 > 0$ as 
\begin{align}\label{DefEntScales}
H_a(g;r_1|r_2) &= H(g; s_{r_1,a}| s_{r_2,a}) = H_a(g;r_1) - H_a(g;r_2) \\ &= (H(gs_{r_1,a}) - H(s_{r_1,a})) - ( H(gs_{r_2,a}) - H(s_{r_2,a})). \nonumber
\end{align} We recall that $\tr(g; r)$ is defined as in \eqref{TraceScaleDef}.

    \section{Preliminaries} \label{PreliminariesSection}

    \subsection{Smoothing functions on $\R^d$}
    
    In this paper we will need to smoothen random variables on $\R^d$ with various smoothing functions. Therefore we introduce the following definition and notation. 
    
    \begin{definition}
    	A family of independent random variables $A = (A_r)_{r > 0}$ is called a \textbf{smoothing family} on $\R^d$ if $A_r$ is an absolutely continuous random variable satisfying that for any $r_1, r_2 > 0$ that $r_{1}^{-1}A_{r_1}$ and $r_2^{-1}A_{r_2}$ have the same distribution.
    \end{definition}
    
    Given a smoothing family $A = (A_r)_{r > 0}$, we define for a random variable $X$ independent of $A$ by 
    \begin{equation}\label{def:smoothened_entropy}
    H^{A}(X; r) = H(X + A_r) - H(A_r). 
    \end{equation}
    
    We prove the following general lemma that apply to all exact dimensional probability measures on $\R^d$ with a polynomial tail decay. We note that by \cite{Feng2023} and Theorem 1.2 of \cite{KittleKogler} all contracting on average self-similar measures satisfy the latter. We denote by $B_R(x)$ the open $R$-ball around $x$ in $\R^d$. 
    
    \begin{lemma} \label{lemma:entropy_from_dimension}
    	Let $(A_r)_{r>0}$ be a family of smoothing functions and given $\eps, r > 0$ define $T_{\eps, r} := \{ x \in \R^d : d(x, \mathrm{supp}(A_r)^C) \geq \varepsilon \}$. Suppose that:
    	\begin{enumerate}[label = (\roman*)]
    		\item There exists some $c > 0$ such that $A_1$ is supported on $B_c(0)$ and the density of $A_1$ is at most $c$ on $B_c(0)$, \\
    		\item For every $\varepsilon > 0$ there exists $\delta > 0$ such that the density of $A_1$ is at least $\delta$ on $T_{\eps, 1}$, \\
    		\item $A_1(T_{\eps, 1}) \to 1$ as $\eps \to 0$ (where we view $A_1$ as a probability measure).
    	\end{enumerate}
    	
    	Suppose that $\lambda$ is an exact dimensional probability measure on $\R^d$ and that there exists some $\alpha > 0$ such that for every sufficiently large $R > 0$ we have $\lambda (B_R^C) < R^{- \alpha}$. Then for every $\varepsilon > 0$ there exists some $r_0 > 0$ such that for every $r \in (0, r_0)$ we have
    	\begin{equation*}
    	H^{A}(\lambda ; r) \in ((\dim \lambda - \varepsilon) \log r^{-1}, (\dim \lambda + \varepsilon) \log r^{-1} ).
    	\end{equation*} 
    \end{lemma}
    
    \begin{proof}
    	We prove this by showing that $$\liminf_{r \to 0} \frac{H^{A}(\lambda ; r)}{\log r^{-1}} \leq \liminf_{r \to 0} \int \frac{\log \lambda(B_r(x))}{\log r} \, d \lambda(x)$$ and $$\limsup_{r \to 0} \frac{H^{A}(\lambda ; r)}{\log r^{-1}} \geq \limsup_{r \to 0} \int \frac{\log \lambda(B_r(x))}{\log r} \, d \lambda(x)$$ and then applying Fatou's lemma.
    	
    	Let $f_r$ be the density function of $\lambda * A_r$. Note that $f_r(x) = \int A_r(u - x) \, d \lambda(u)$ and in particular
    	\begin{equation}
    	f_r(x) \leq r^{-d} \lambda(B_{cr}(x)) \label{eq:density_bounds}
    	\end{equation}
    	
    	First we find a lower bound for $H(\lambda * A_r)$. Note by \eqref{eq:density_bounds}
    	\begin{align*}
    	H(\lambda * A_r) &= \int \int - \log f_r(x + y) \, d A_r(y) \, d \lambda(x). \\
    	& \geq \int \int  - \log \lambda( B_{cr}(x + y))\, d A_r(y) \, d \lambda(x) + d \log r^{-1} - \log c\\
    	& \geq \int - \log \lambda(B_{r c}(x))\, d \lambda(x) + d \log r^{-1} - \log c.
    	\end{align*}
    	In particular
    	\begin{equation*}
    	H(\lambda ; r) \geq \int - \log \lambda(B_{r c}(x)) \, d \lambda(x) - \log c
    	\end{equation*}
    	and so by Fatou's lemma and exact dimensionality,
    	\begin{equation*}
    	\liminf_{r \to 0} \frac{H(\lambda ; r)}{\log r^{-1}} \geq \int \liminf_{r \to 0}  \frac{\log \lambda(B_{c r}(x))}{\log r} \, d \lambda(x) = \dim \lambda.
    	\end{equation*}
    	
    	Next we find our upper bound for $H(\lambda * A_r)$. First let $S_0 = B_1(0)$ and for $n \geq 1$ let $S_n = B_{2^n}(0) \backslash B_{2^{n-1}}(0)$. Let $\lambda_n := \lambda |_{S_n}$ and let $S_n' := B_1(A_n)$. By \cite{KittleKoglerEntropy}*{Lemma 2.2} we have that
    	\begin{equation*}
    	H(\lambda * A_r) \leq \sum_{n=0}^{\infty} H(\lambda_n * A_r) + \sum_{n=0}^{\infty} - \| \lambda_n\|_1 \log \| \lambda_n \|_1.
    	\end{equation*}
    	Let $f_{n, r}$ be the density function of $\lambda_n * A_r$. We now bound $H(\lambda_n * A_r)$. We have
    	\begin{align*}
    	H(\lambda_n * A_r) &= \int \int - \log f_{n, r} (x + y) \, d A_r(y) \, d \lambda_n(x) + \| \lambda_n \|_1 \log \| \lambda_n \|_1.
    	\end{align*}
    	Let  $\eps > 0$ and choose $\delta > 0$ such that the density of $A_1$ is at least $\delta$ on $T_{\eps, 1}$. Note that for all $x \in A_n$ and $y \in T_{2 \eps, r}$ we have
    	\begin{align*}
    	f_{n, r}(x + y) &= \int A_r(x + y - u) \, d \lambda_n(u)  \\
    	& \geq \int_{B_{\eps r}(x)} A_r(x + y - u) \, d \lambda_n(u) \\
    	& \geq \int_{B_{\eps r}(x)} \delta r^{-d} \, d \lambda_n(u) \\
    	& = \delta r^{-d} \lambda_n (B_{\eps r}(x))
    	\end{align*}
    	Therefore
    	\begin{align*}
    	\MoveEqLeft \int_{S_n} \int_{T_{2 \eps, r}} - \log f_{n, r} (x + y) \, d A_r \, d \lambda_n(x)\\ &\leq A_1(T_{2 \eps, 1}) \int - \log \delta \lambda_n (B_{\eps r} (x)) \, d \lambda_n(x) - A_1(T_{2 \eps, 1}) \| \lambda_n \| d \log r^{-1}.
    	\end{align*}
    	Now let $g_{n, r}$ be the density function of $\lambda_n * ( A_r |_{T_{2 \eps, r}^C})$. Clearly $g_{n, r} \leq f_{n, r}$ and so
    	\begin{align}
    	\MoveEqLeft \int_{S_n} \int_{T_{2 \eps, r}^C} - \log f_{n, r} (x + y) \, d A_r \, d \lambda_n(x) \nonumber \\ 
    	&= \int_{S_n'} -g_{n, r}(u) \log f_{n, r}(u) \, du \nonumber\\
    	&\leq \int_{S_n'} - g_{n, r}(u) \log g_{n, r}(u) \, du \nonumber\\
    	& \leq - \| g_{n, r} \|_1 \log \frac{\| g_{n, r} \|_1}{\mathrm{vol}_{\R^d}(S_n')} \label{eq:some_jensen_estimate} \\
    	& = A_1(T_{2 \eps, 1}^C) \| \lambda_n \| \log \frac{\mathrm{vol}_{\R^d}(S_n')}{ A_1(T_{2 \eps, 1}^C) \| \lambda_n \|} \nonumber
    	\end{align}
    	where \eqref{eq:some_jensen_estimate} follows from Jensen's inequality. Putting this together we get
    	\begin{align*}
    	\MoveEqLeft H(\lambda_n * A_r) \\ & \leq A_1(T_{2 \eps, 1}) \int - \log \lambda_n (B_{\varepsilon r} (x)) \, d \lambda_n(x) - A_1(T_{2 \eps, 1})\| \lambda_n \| d \log r^{-1}\\&  \,\,\,\,\,\,\,\,\, + A_1(T_{2 \eps, 1}) \log \delta^{-1} = A_1(T_{2 \eps, 1}^C)\| \lambda_n \| \log \frac{\mathrm{vol}_{\R^d}(S_n')}{ A_1(T_{2 \eps, 1}^C) \| \lambda_n \|} + \| \lambda_n \| \log \| \lambda_n \|.
    	\end{align*}
    	Therefore
    	\begin{equation*}
    	\limsup_{r \to 0} \frac{H(\lambda_n * A_r )}{\log r^{-1}} \leq A_1(T_{2 \eps, 1}) \left[ \limsup_{r \to 0} \int \frac {\log \lambda_n(B_{\eps r}(x))}{\log r} \, d \lambda_n(x) - d \| \lambda_n \|_1 \right] 
    	\end{equation*}
    	Clearly $\lambda_n$ is exact dimensional with the same dimension as $\lambda$ and so by Fatou's lemma we get
    	\begin{equation*}
    	\limsup_{r \to 0} \int \frac {\log \lambda_n(B_{\eps r}(x))}{\log r} \, d \lambda_n(x) \leq \| \lambda_n \|_1 \dim \lambda.
    	\end{equation*}
    	Noting that $\eps$ can be arbitrarily small we deduce that
    	\begin{equation*}
    	\limsup_{r \to 0} \frac{H(\lambda_n * A_r )}{\log r^{-1}} \leq \| \lambda_n \|_1 (\dim \lambda - d).
    	\end{equation*}
    	
    	Finally by Fatou's lemma we get
    	
    	\begin{align*}
    	\limsup_{r \to 0} \frac{H(\lambda * A_r )}{\log r^{-1}} &\leq \limsup_{r \to 0}\sum_{n=0}^{\infty} \frac{H(\lambda_n * A_r )}{\log r^{-1}} \\
    	& \leq \sum_{n=0}^{\infty} \limsup_{r \to 0}\frac{H(\lambda_n * A_r )}{\log r^{-1}} \\
    	& \leq \sum_{n=0}^{\infty} \| \lambda_n \|_1 (\dim \lambda - d) \\
    	& = \dim \lambda - d.
    	\end{align*}
    	
    	The result follows.
    	
    \end{proof}
    
    In order to prove Proposition~\ref{proposition:small_scale_maximum_entropy}, we will use smoothening by a uniform probability measure on a cube. Indeed, denote by $\xi = (\xi_r)_{r > 0}$ the family of smoothing functions where $\xi_r$ is the uniform probability measure on $[-r/2, r/2]^d$. Then as in \eqref{def:smoothened_entropy} we define 
    \begin{equation}\label{def:cube_smoothing}
    H^{\xi}(\lambda; r) = H(\lambda * \xi_r) - H(\xi_r).
    \end{equation}
    We then have the following result that will be used in the proof of Proposition~\ref{proposition:small_scale_maximum_entropy}.
    
    \begin{lemma} \label{lemma:close_means_close_entropy_at_scale}
    	Suppose that $r > 0$ and that $X$ and $Y$ are random variables taking values in $\R^d$ such that $|X - Y| < C r$ almost surely. Then 
    	\begin{equation*}
    	\left| H^{\xi}(X ; r) - H^{\xi}(Y ; r) \right| < d \log (2 C + 4)
    	\end{equation*}
    \end{lemma}

    \begin{proof}
    	Let $U$ and $V$ be two independent uniform random variables on $[-r/2, r/2]^d$ which are independent of $(X, Y)$. First note that
    	\begin{align*}
    	H^{\xi}(Y ; r) - H^{\xi}(X ; r) &= H(Y + V) - H(X + U) \\
    	&= (H(Y+V, X + U) - H(X + U)) \\ &- (H(Y+V, X+U) - H(Y+U)) \\
    	&= H(Y+V | X+U) - H(X+U | Y + V).
    	\end{align*}
    	Clearly $Y+V$ is contained in a hypercube with side length $(2C + 4) r$ and centre $X+U$. Therefore $H(Y+V | X+U) \leq d \log (2C + 4) + d \log r$. Also $H(X + U | Y+V) \geq H(U) = d \log r$. The result follows.
    \end{proof}
    
    \subsection{Gaussian Approximation and Full Dimension}
    
    The aim of this subsection is to prove the following proposition, which will be used in section~\ref{SectionConclusion}. Given a random variable $x$ and a $\sigma$-algebra $\mathscr{A}$ we denote by $x|\mathscr{A}$ the regular conditional distribution as defined and discussed in \cite{KittleKoglerEntropy}*{Section 2.3}, which we note is a Markov kernel. We furthermore recall that a probability measure $\lambda$ on $\R^d$ is called exact dimensional if there exists some $\alpha \in [0,d]$ such that for $\lambda$-almost all $x \in \R^d$ we have that as $r \to 0$, $$\lambda(\{y \in \R^d \,:\, |y-x| < r \}) = r^{\alpha + o_{\mu,x}(r)}.$$  
    
    \begin{proposition} \label{propositon:normal_to_entropy}
    	For every $\varepsilon > 0$ there is some $C>0$ such that the following holds. Suppose that $\lambda$ is an exact dimensional measure on $\R^d$ and that for all sufficiently small $r>0$ we can construct a sample $x$ from $\lambda$, some $\sigma$-algebra $\mathscr{A}$ and a $\mathscr{A}$-measurable random positive semi-definite symmetric matrix $\Sigma$ as well as some $\mathscr{A}$-measurable random $x_0 \in \R^d$ such that with probability at least $1 - C^{-1}$ we have
    	\begin{equation*}
    	\Sigma \geq Cr^2 I \quad \text{and} \quad \mathcal{W}_1(x|\mathcal{A}, N(x_0, \Sigma)) < C^{-1}r.
    	\end{equation*}
    	Then $\dim \lambda \geq d - \varepsilon$.
    \end{proposition}
    
    In order to prove Proposition \ref{propositon:normal_to_entropy} we first need some estimates on entropy. We denote by $\mathcal{TV}$ the total variation distance. 
    
    \begin{lemma} \label{lemma:lipshitz_wasserstien_tv}
    	Suppose that $A$ is a random variable with $c$-Lipshitz density function which is independent from $(X, Y)$. Then
    	\begin{equation*}
    	\mathcal{TV} (X + A, Y + A) \leq \frac{1}{2}c \mathcal{W}_1(X, Y).
    	\end{equation*}
    \end{lemma}
    
    \begin{proof}
    	Let $f$ be the density function of $A$. Note that the density functions of $X$ and $Y$ at $x$ are given by $\mathbb{E}[f(x-X)]$ and $\mathbb{E}[f(x-Y)]$ respectively. Hence
    	\begin{align*}
    	\mathcal{TV}(X+A, Y+A) &= \frac{1}{2} \int \left| \mathbb{E}[f(x-X) - f(x-Y)] \right| \, dx \\
    	&\leq \frac{1}{2} \int  \mathbb{E}\big[ |f(x-X) - f(x-Y)| \big]  \, dx \\
    	& \leq \frac{1}{2} c \mathbb{E}\big[|X-Y|\big].
    	\end{align*}
    	The result now follows by choosing couplings between $X$ and $Y$ such that $\mathbb{E}\big[|X-Y|\big] \to \mathcal{W}_1(X, Y)$ and by noting that the total variation distance does not depend on the choice of coupling.
    \end{proof}
    
    \begin{lemma} \label{lemma:weird_tv_entropy_estimate}
    	Suppose that $A \subset \R^d$ has finite Lebesgue measure and that $g, h$ are integrable functions from $A$ to $[0, e^{-1}]$. Let $\varepsilon = \int_A |g - h| \, d\Haarof{\R}$. Then $$ \left| \int_A g \log g - h \log h \right| \, d\Haarof{\R} \leq - \varepsilon \log \varepsilon + \varepsilon \log m(A).$$
    \end{lemma}
    
    \begin{proof}
    	First note $\left| \int_A g \log g - h \log h  \, d\Haarof{\R} \right| \leq \int_A \left| g \log g - h \log h \right| \, d\Haarof{\R}   $. By looking at the derivative of $x \log x$ it is straightforward to see that $\left| g \log g - h \log h \right| \leq - |g - h| \log | g - h|$. In other words for fixed $|g-h|$ the left hand side is maximized when one of $g$ and $h$ is zero. Integrating this and applying Jensen's inequality proves the lemma.
    \end{proof}
    
    \begin{lemma} \label{lemma:cuboid_uniform}
    	For every family of smoothing functions $A$ such that $A_1$ has bounded Lipshitz density and satisfies the conditions of Lemma \ref{lemma:entropy_from_dimension}, every $c > 1$ and every $\varepsilon > 0$ there is some $C > 0$ such that the following holds. Let $r>0$ and suppose that $\lambda_1$ is the uniform probability measure on a hyper cuboid of side length $Cr$. Suppose that $A_1$ is Lipshitz and that $\mathcal{W}_1(\lambda_1, \lambda_2) \leq C^{-1} r$. Then $$H^{(A)}(\lambda_2 ; r | cr) \geq d \log c - \varepsilon.$$
    \end{lemma}
    
    \begin{proof}
    	Clearly this statement is independent of $r$. We let $C \to \infty$ and at the same time let $r \to 0$ in such a way that $Cr$ is constant. Clearly $H^{(A)}(\lambda_1 ; r | cr)$ tends to $d \log c$ as both $H(\lambda_1 * A_r)$ and $H(\lambda_1 * A_{cr})$ will tend to $H(\lambda_1)$ by the dominated convergence theorem. This means that providing $C$ is sufficiently large for all $r>0$ we have $H^{(A)}(\lambda_1 ; r | cr) \geq d \log c - \varepsilon$. By fixing $r$ large enough that the density function of $A_r$ is at most $e^{-1}$, letting $C \to \infty$ and applying Lemma \ref{lemma:lipshitz_wasserstien_tv} and Lemma \ref{lemma:weird_tv_entropy_estimate} we can show that $H^{(A)}(\lambda_2; r) - H^{(A)}(\lambda_1; r) \to 0$ and $H^{(A)}(\lambda_2; cr) - H^{(A)}(\lambda_1; cr) \to 0$. This completes the proof.
    \end{proof}
    
    This is enough to prove Proposition \ref{propositon:normal_to_entropy}.
    
    \begin{proof}[Proof of Proposition \ref{propositon:normal_to_entropy}]
    	We let $A_r$ be the normal distribution with mean zero and standard deviation $r$. By the convexity of entropy between two scales and rescaling it is sufficient to show that for every $\varepsilon > 0$ there exists some $C$ such that if $\Sigma \geq CI$ and $\mathcal{W}_1(\lambda, N(0, \Sigma)) \leq C^{-1}$ then $H^{(A)}(\lambda ; 1 | 2) \geq d \log 2 - \varepsilon$. We choose some large $C_1$ and divide $\R^d$ into hypercuboids of side length $C_1$. Let $\mathcal{P}$ denote this partition. By the convexity of entropy between two scales it is sufficient to show that
    	\begin{equation*}
    	\sum_{B \in \mathcal{P}} \lambda(B) H^{(A)}\left( \frac{\lambda}{\lambda(B)} ; 1 | 2 \right) \geq d \log 2 - \varepsilon.
    	\end{equation*}
    	This follows by applying Lemma \ref{lemma:cuboid_uniform}, choosing $C_1$ to be sufficiently large in terms of $\varepsilon$ and letting $C$ be sufficiently large in terms of $\varepsilon$ and $C_1$.
    \end{proof}

    \subsection{Taylor Expansion Bound}
    
    We state the following Taylor expansion bound form \cite{KittleKogler}, that will be used in section~\ref{SectionConclusion}. This bound relies on the $G = \mathrm{Sim}(\R^d)$ action on $\R^d$ having no second derivatives.

    \begin{proposition}(\cite{KittleKogler}{Proposition 3.4})\label{MainTaylorBound}
    	For every $A > 0$ there exists $C = C(d,A) > 1$ such that the following holds. Let $n \geq 1$, $r \in (0,1)$ and let $u^{(1)}, \ldots , u^{(n)} \in \mathfrak{g}$. Let $g_1, \ldots , g_n \in G$ with $$\rho(g_i) < 1, \quad \quad |b(g_i)| \leq A \quad \text{ and } \quad |u^{(i)}| \leq \rho(g_1 \cdots g_i)^{-1}r < 1.$$ Let $v \in \R^d$ with $|v| \leq A$ and write $$x = g_1 \exp(u^{(1)})\cdots g_n\exp(u^{(n)})v$$ and $$\zeta_i = D_0(g_1g_2 \cdots g_{i}\exp(u)g_{i + 1} \cdots g_{n-1}g_n v)$$ and let $$S = g_1\cdots g_n v + \sum_{i = 1}^n \zeta_i(u^{(i)}).$$ Then it holds that $$|x-S| \leq C^n\rho(g_1 \cdots g_n)^{-1}r^2.$$
    \end{proposition}

    \section{Entropy and Trace Bounds} \label{SectionEntropy}

   \subsection{Results from \cite{KittleKoglerEntropy}}\label{section:EntropyResults}
   
   For the convenience of the reader, we recall some results from \cite{KittleKoglerEntropy} that will be used to deduce suitable entropy bounds. We refer to \cite{KittleKoglerEntropy} for a discussion of these results. 
   
   \begin{definition}\label{EtaLDPDef}
   	Let $\eta = (\eta_n)_{n \geq 1}$ be a sequence of stopping times. Then we say that $\eta$ satisfies the large deviation  principle if $\E[\eta_n] \to \infty$ as $n \to \infty$ and for every $\eps > 0$ there exists a $\delta > 0$ such that for all sufficiently large $n$, $$\mathbb{P}\big[|\eta_n - \mathbb{E}[  \eta_n]| \geq \eps \cdot  \mathbb{E}[ \eta_n] \big] \leq e^{-\delta \cdot \mathbb{E}[\eta_n]}.$$
   \end{definition}
   
   \begin{theorem}(\cite{KittleKoglerEntropy}*{Theorem 1.2})\label{Cor2EntStoppedRW}
   	Let $\mu$ be a finitely supported probability measure on $G$. Let $\eta_n = (\eta_n)_{n \geq 1}$ be a sequence of stopping times satisfying the large deviation principle and denote $L_n = \mathbb{E}[\eta_n]$ for $n\geq 1$. Let $a \geq 1$, $\eps > 0$ and let $r_n > 0$ be a sequence satisfying for all $n\geq 1$, $$r_n \leq a^{-1}c_G M_{\lceil (1 + \eps)L_n \rceil}$$ for a constant $c_G > 0$ depending only on $G$. Then for all $n \geq 1$, 
   	$$H_a(q_{\eta_{n}} ; r_n) \geq h_{\mu}\cdot L_n + O_{\mu,\eta}(\eps L_n).$$
   \end{theorem}

   \begin{proposition} \label{EntropyGapToTrSum}(\cite{KittleKoglerEntropy}*{Proposition 1.5})
   	Let $g$ be a $G$-valued random variable independent of $(s_{a,r})_{a \geq 1, r > 0}$ and let $0 < r_1 < r_2$. Let $a \geq 1$ be such that $ar_2$ is sufficiently small in terms of $G$. Suppose that for all $r_1' \in [r_1, 2r_1]$ as well as $r_2' \in [r_2/2, 2r_2]$ it holds for some constant $C> 0$ that $$H_a(g; r_1'| r_2') \geq C.$$ Let $A > 1$. Then there exists $s_1, \dots, s_{m} \in (ar_1, 4ar_2)$ where $m = \lceil \frac{\log 4ar_2 - \log ar_1}{2\log A} \rceil$ such that for $N = \left\lceil \frac{\log r_2 - \log r_1}{\log 2} \right\rceil - 1$,
   	\begin{equation*}
   	\sum_{i=1}^m \tr (g ; s_i) \gg_G \frac{C - N\cdot O_G(e^{-\frac{a^2}{4}} + a^3r_2)}{a^2 \log A}
   	\end{equation*}
   	and $s_{i+1} \geq A s_i$ for all $1 \leq i \leq m-1$.
   \end{proposition}
   
   \subsection{Entropy Gap and Trace Bound for Theorem~\ref{GeneralHochman}}  
   
   We first show the following entropy gap and then use it to deduce a suitable trace bound.
   
   \begin{proposition}\label{proposition:entropy_gap_first}
   	Let $\mu$ be a finitely supported, contracting on average probability measure on $G$. Suppose that for some $c > 0$ we have $M_{\ell} \geq e^{-c\ell}$ for infinitely many $\ell\geq 1$. Assume further that $\dim \nu < \frac{h_{\mu}}{|\chi_{\mu}|}$. Then there exist constants $\alpha_1, \beta, \delta > 0$ depending on $\mu$ with $\beta > \delta$ and a sequence $\kappa_n \to 0$ such that for $a \geq 1$ and $r_n^{(1)}, r_n^{(2)} > 0$, $$H_a(q_{\tau_{\kappa_n}} ; r_n^{(1)}| r_n^{(2)}) \geq \alpha_1 \log \kappa_n^{-1}.$$  for all sufficiently large $n$ with $$r_n^{(1)} \leq \kappa_n^{\beta} \quad\quad \text{ and } \quad\quad r_n^{(2)} \in [\kappa_n^{\delta}/2,2\kappa_n^{\delta}].$$
   \end{proposition}
   
   We show Proposition~\ref{proposition:entropy_gap_first} by establishing the following two lemmas.
   
   \begin{lemma} \label{lemma:separation_minimal_entropy}
   	Let $\mu$ be a finitely supported, contracting on average probability measure on $G$. Suppose that for some $c > 0$ we have $M_{\ell} \geq e^{-c\ell}$ for infinitely many $\ell\geq 1$. Then there exists $\beta > 0$ and a sequence $\kappa_n \to 0$ such that for $a \geq 1$ and $\eps > 0$ and sufficiently large (depending on $a$ and $\eps$) $n$, $$H_a(q_{\tau_{\kappa_n}} ; r_n) \geq \left( \frac{h_{\mu}}{|\chi_{\mu}|} - \eps \right) \log \kappa_n^{-1} \quad\quad \text{ for any }\quad\quad r_n \leq \kappa_n^{\beta}.$$
   \end{lemma}
   
   \begin{proof}
   	If $M_{\ell} \geq e^{-c\ell}$ then $M_{k} \geq e^{-2ck}$ for $k \in \{ \lceil \ell/2 \rceil, \ldots , \ell \}$. So we simply choose a decreasing sequence of $\kappa_n$ such for sufficiently large $n$ we have that $\lceil \E[\tau_{\kappa_n}](1 + \eps) \rceil \in \{ \lceil \ell_n/2 \rceil, \ldots , \ell_n \}$ for the given infinite increasing sequence $\ell_n \geq 1$ with $M_{\ell_n} \geq e^{-c\ell_n}$.
   	Note that $\E[\tau_{\kappa}] = \frac{\log \kappa^{-1}}{|\chi_{\mu}|} + o_{\mu}(\log \kappa^{-1})$ and so it follows that for a suitably chosen $\beta > 0$ depending on $\mu$  and for $n$ sufficiently large $$\kappa_n^{\beta} \leq a^{-1}c_G M_{\lceil (1 + \eps)\E[\tau_{\kappa_n}] \rceil}.$$ Thus the lemma follows by Theorem~\ref{Cor2EntStoppedRW}, using that $\tau_{\kappa}$ satisfies the large deviation principle by Lemma 3.11 from \cite{KittleKogler}.
   \end{proof}

   \begin{proposition} \label{proposition:small_scale_maximum_entropy}
   	Let $\mu$ be a finitely supported, contracting on average probability measure on $G$. Then for every $\varepsilon>0$ there is some $\delta>0$ such that whenever $\kappa > 0$, $a \geq 1$ and $\kappa^{\delta} a$ is sufficiently small (in terms on $\mu$ and $\varepsilon$) we have $$H_a(q_{\tau_{\kappa}}; \kappa^{\delta} ) \leq \left( \dim \nu + \varepsilon \right) \log \kappa^{-1}.$$
   \end{proposition}

   \begin{proof}[Proof of Proposition \ref{proposition:small_scale_maximum_entropy}]
   	The proof is similar to Lemma 7.3 from \cite{KittleKogler}. First note that the Haar measure on $G$, which we denote by $\Haarof{G}$, can be written as
   	\begin{equation*}
   	\int f(g) \, d \Haarof{G} (g) = \int \int \int f(\rho U + b) \rho^{-(d+1)} \, d \rho d U d b.
   	\end{equation*}
   	In other words it can be expressed as a product measure $\rho^{-(d+1)} d \rho d U d b$. We will find an upper bound on $H(q_{\tau_{\kappa}} s_{\kappa^{\delta}, a})$ by using \cite{KittleKoglerEntropy}{Lemma 2.5}.  We provide an upper bound on the entropy of $q_{\tau_{\kappa}} s_{\kappa^{\delta}, a}$ under the natural projections to each of $(\R_{>0}, \mathcal{B}(\R), \rho^{-(d+1)} d \rho)$, $(O(d), \mathcal{B}(O(d)), \Haarof{O(d)})$, and $\R^d$. Note that $O(d)$ is compact so has finite Haar measure and therefore $$H(U(q_{\tau_{\kappa}} s_{\kappa^{\delta}, a})) \leq O(1).$$

   	Clearly providing $\kappa^{\delta} a$ is sufficiently small we have $\rho(q_{\tau_{\kappa}} s_{\kappa^{\delta}, a}) \in [R^{-1} \kappa /2, 2 \kappa]$. The $\rho^{-d-1} d \rho$ measure of this interval is at most $O(R^{d+1} \kappa^{-d})$ so $H(\rho(q_{\tau_{\kappa}} s_{\kappa^{\delta}, a})) \leq d \log \kappa^{-1} + O(\log R) + O(1)$.
   	
   	Finally we need to bound $H(b(q_{\tau_{\kappa}} s_{\kappa^{\delta}, a}))$. We introduce the following family of smoothing functions. As in \eqref{def:cube_smoothing}, denote by $\xi = (\xi_r)_{r > 0}$ the family of smoothing functions where $\xi_r$ is the uniform probability measure on $[-r/2, r/2]^d$ and recall the notation $H^{\xi}(\lambda, r)$.
   	
   	Choose $T$ large enough that $\nu \{ x : |x| \geq T \} \leq \varepsilon$ and let $\nu_T = \nu|_{x:|x| \leq T}$. By \cite{KittleKoglerEntropy}*{Lemma 2.1} we have
   	$$H((\mu^{*\tau_{\kappa}}*\nu_T)* \xi_{\kappa}) + H((\mu^{*\tau_{\kappa}}*(\nu - \nu_T)) * \xi_{\kappa}) \leq H(\nu * \xi_{\kappa})$$
   	By Lemma \ref{lemma:entropy_from_dimension} providing $\kappa$ is sufficiently small we have $H^{\xi}(\nu ; \kappa) \leq (\dim \nu + \varepsilon) \log \kappa^{-1}$ and so $H(\nu * \xi_{\kappa}) \leq (\dim \nu + \varepsilon) \log \kappa^{-1} + H(\xi_{\kappa}) = (\dim \nu + \varepsilon - d) \log \kappa^{-1}$.
   	
   	Note that $H((\mu^{*\tau_{\kappa}}*(\nu - \nu_T))* \xi_{\kappa}) \geq \| \nu - \nu_{T} \|_1 H(\xi_{\kappa}) \geq - \varepsilon d \log \kappa^{-1}$. This means
   	\begin{equation*}
   	H((\mu^{*\tau_{\kappa}}*\nu_T)* \xi_{\kappa}) \leq (\dim \nu + \varepsilon - d + d \varepsilon) \log \kappa^{-1}
   	\end{equation*}
   	and so
   	\begin{equation*}
   	H((\mu^{*\tau_{\kappa}}*\frac{\nu_T}{\| \nu_T\|_1})* \xi_{\kappa}) \leq (\dim \nu - d + O_d(\varepsilon)) \log \kappa^{-1}.
   	\end{equation*}
   	This means
   	\begin{equation*}
   	H^{\xi}( \mu^{*\tau_{\kappa}}*\frac{\nu_T}{\| \nu_T\|_1} ; \kappa) \leq (\dim \nu + O_d(\varepsilon)) \log \kappa^{-1}
   	\end{equation*}
   	and so by Lemma \ref{lemma:close_means_close_entropy_at_scale} we have that $$H^{\xi}( b(q_{\tau_{\kappa}} s_{\kappa^{\delta}, a});\kappa) \leq (\dim \nu + O_d(\varepsilon)) \log \kappa^{-1} + O(\log (T + a + 1)).$$ In particular providing $\kappa^{\delta} a$ is sufficiently small we have $$H^{\xi}( b(q_{\tau_{\kappa}} s_{\kappa^{\delta}, a});\kappa) \leq (\dim \nu + O_{\mu}(\delta + \varepsilon)) \log \kappa^{-1}.$$
   	Noting that $H(b(q_{\tau_{\kappa}} s_{\kappa^{\delta}, a})) \leq H(b(q_{\tau_{\kappa}} s_{\kappa^{\delta}, a}) * \xi_{\kappa})$ we can therefore conclude
   	\begin{equation*}
   	H(b(q_{\tau_{\kappa}} s_{\kappa^{\delta}, a})) \leq (\dim \nu + O_{\mu}(\delta + \varepsilon) - d) \log \kappa^{-1}.
   	\end{equation*}
   	Putting these estimates together we get that providing $\kappa^{\delta} a$ is sufficiently small
   	\begin{equation*}
   	H(q_{\tau_{\kappa}} s_{\kappa^{\delta}, a}) \leq (\dim \nu + O_{\mu}(\delta + \varepsilon) ) \log \kappa^{-1}.
   	\end{equation*}
   	The result follows by noting that $H(s_{\kappa^{\delta}, a}) \geq -O_{\mu}(\delta  \log \kappa^{-1})$ and that $\delta$ and $\varepsilon$ can be taken to be arbitrarily small.
   \end{proof}
   
   By combining Lemma~\ref{lemma:separation_minimal_entropy} and Proposition~\ref{proposition:small_scale_maximum_entropy}, Proposition~\ref{proposition:entropy_gap_first} follows with setting for example $\alpha_1 = \frac{1}{2}(\frac{h_{\mu}}{|\chi_{\mu}|} - \dim \nu)$. We use Proposition~\ref{proposition:entropy_gap_first} to deduce the following trace bound.
   
   \begin{proposition}\label{proposition:trace_bound_first}
   	Under the assumptions and with the notation of Proposition~\ref{proposition:entropy_gap_first}, there exists a constant $\alpha_2 = \alpha_2(\mu) > 0$ such that the following holds. Let $a \geq 1$ be sufficiently large. Then there exists a decreasing sequence $\kappa_n \to 0$ and $r_n \in (\kappa_n^{\beta}, 2a\kappa_n^{\delta})$ such that $$\mathrm{tr}(q_{\tau_{\kappa_n}}; r) \geq \alpha_2.$$
   \end{proposition}
   
   \begin{proof}
   	This follows from Proposition~\ref{proposition:entropy_gap_first} and Proposition~\ref{EntropyGapToTrSum}. Indeed we may set for example $A = 2$. Then $N = O(\log \kappa_n^{-1})$ and so for $a$ sufficiently large and $m = O(\log 2a \kappa_n^{\delta} - \log \kappa_n^{\beta})$ (which satisfies  $a = O_{\delta, \beta}(\log \kappa_n^{-1})$ for $n$ sufficiently large) it follows by Proposition~\ref{EntropyGapToTrSum} that there exist $s_1, \ldots , s_m \in (\kappa_n^{\beta}, 2a\kappa_n^{\delta})$ such that $$\sum_{i=1}^m \tr (g ; s_i) \gg_G \alpha_1 \log \kappa^{-1}.$$ The claim follows by choosing $r_n = s_i$ such that $\tr (g ; s_i) = \max_{1 \leq j \leq m} \tr (g ; s_j)$.
   \end{proof}
   
   \subsection{Entropy Gap and Trace Bound for Theorem~\ref{MainWeakHochman}}
   
   Towards Theorem~\ref{MainWeakHochman} we establish the following entropy and trace bounds. 
   
   \begin{lemma}\label{lemma:entropy_gap_second}
   	Let $\mu$ be a finitely supported, contracting on average probability measure on $G$. Suppose that for some $B > 0$ we have $$\log M_n \geq -n\exp((\log n)^{B})$$ for all $n\geq 1$. Assume further that $\dim \nu < \frac{h_{\mu}}{|\chi_{\mu}|}$. Then there is a constant $\delta > 0$ depending on $\mu$ such that for all $\eps_1 > 0$ the following holds for sufficiently small $\kappa$. Then for $a \geq 1$  and $r_1, r_2 > 0$, $$H_a(q_{\tau_{\kappa}} ; r_1| r_2) \geq \alpha \log \kappa^{-1}$$ assuming 
   	$$r_1 \leq a^{-1}\kappa^{\exp((\log \log \kappa^{-1})^{B + \eps_1})} \quad\quad \text{ and } \quad\quad a r_2 \leq \kappa^{\delta}.$$
   \end{lemma}
   
   \begin{proof}
   	The proof is analogous to the one of Proposition~\ref{proposition:entropy_gap_first}. The $\eps_1 > 0$ is there to ensure that for sufficiently small $\kappa$, $$\kappa^{\exp((\log \log \kappa^{-1})^{B + \eps_1})} \leq c_GM_{\lceil (1 + \eps)\mathbb{E}[\tau_{\kappa}] \rceil}$$ for $c_G$ the constant from Proposition~\ref{Cor2EntStoppedRW}. The claim then follows from Proposition~\ref{Cor2EntStoppedRW} and Proposition~\ref{proposition:small_scale_maximum_entropy}.
   \end{proof}

   \begin{proposition}\label{proposition:trace_sum_second} 
   	Let $\mu$ be a finitely supported, contracting on average probability measure on $G$. Suppose that for some $B > 0$ we have $$\log M_n \geq -n\exp((\log n)^{B})$$ for all $n\geq 1$. Assume further that $\dim \nu < \frac{h_{\mu}}{|\chi_{\mu}|}$.
   	
   	Then for every $\eps_1 > 0$ the following holds for $\kappa$ sufficiently small: There exist $s_1, \ldots, s_{m} > 0$ with $m$ and integer satisfying $m \asymp \exp((\log \log \kappa^{-1})^{B + \eps_1}))$ such that for each $1\leq i \leq m$, we have that
   	$$s_i \in (\kappa^{\exp((\log \log \kappa^{-1})^{B + \eps_1})}, 2(\log \log \kappa^{-1})^{\frac{B}{2} + \eps_1} \kappa^{\delta}),$$
   	for each $1 \leq i \leq m-1$ we have $s_{i + 1} \geq \kappa^{-3}s_i$ and  it holds that $$\sum_{i = 1}^m \mathrm{tr}(q_{\tau_{\kappa}}; s_i) \geq \frac{\alpha}{(\log \log \kappa^{-1})^{B + 2\eps_1}}.$$
   \end{proposition}
   
   \begin{proof}
   	Let $a \geq 1$ be to be determined. Then  by Lemma~\ref{lemma:entropy_gap_second} and Proposition~\ref{EntropyGapToTrSum} for $A = \kappa^{-3}$ and for every $\eps_1 > 0$ we have that $N \asymp (\log \kappa^{-1}) \exp((\log \log \kappa^{-1})^{B + \eps_1})$.  In order for the numerator in Proposition~\ref{EntropyGapToTrSum} to be $\geq \alpha \log \kappa^{-1}$ for a changed constant $\alpha$ we require that $$Ne^{-a^2/4} \leq c \log \kappa^{-1}$$ for a sufficiently small constant $c > 0$, which is equivalent to $a^2 \geq  (\log \log \kappa^{-1})^{B + \eps_1}$. We therefore set $$a^2 = (\log \log \kappa^{-1})^{B + 2\eps_1},$$ concluding the proof by Proposition~\ref{EntropyGapToTrSum} with $m \asymp \exp((\log \log \kappa^{-1})^{B + \eps_1}))$.
   \end{proof}

    \section{Decomposition of Stopped Random Walks and Variance Summation}
    \label{VarianceSummation}
    
    \subsection{Proper Decompositions}
    
    We recall the following from \cite{KittleKogler} in order to disintegrate our self-similar measure into measures which are approximately a sum of small independent random variables.
    
    \begin{definition}\label{ProperDecomp}
    	Let $\mu$ be a probability measure on $G$, let $n,K \in \Z_{\geq 0}$ and let $A, r > 0$ and $r \in (0,1)$. Then a \textbf{proper decomposition} of $(\mu,n,K,A)$ at scale $r$ consists of the following data
    	\begin{enumerate}[label = (\roman*)]
    		\item $f = (f_i)_{i = 1}^n$ and $h = (h_i)_{i = 1}^n$ random variables taking values in $G$,
    		\item $U = (U_i)_{i = 1}^n$ random variables taking values in $\mathfrak{g}$,
    		\item $\mathscr{A}_0 \subset \mathscr{A}_1 \subset \ldots \subset \mathscr{A}_n$ a nested sequence of $\sigma$-algebras,
    		\item $\gamma = (\gamma_i)_{i = 1}^{\infty}$ be i.i.d. samples from $\mu$ and let $\mathscr{F} = (\mathscr{F}_i)_{i = 1}^{\infty}$ be a filtration for $\gamma$ with $\gamma_{i + 1}$ being independent from $\mathscr{F}_i$ for $i \geq 1$,
    		\item stopping times $S = (S_i)_{i = 1}^n$ and $T = (T_i)_{i = 1}^n$ for the filtration $\mathscr{F}$,
    		\item $m = (m_i)_{i = 1}^n$ non-negative real numbers,
    	\end{enumerate}
    	
    	satisfying the following properties:
    	
    	\begin{enumerate}[label=\textbf{A\arabic*}]
    		\item The stopping times satisfy $$S_1 \leq T_1 \leq  S_2 \leq T_2 \leq \ldots \leq S_n \leq T_n,$$ $S_1 \geq K$ as well as $S_i \geq T_{i-1} +K$ and $T_i \geq S_i + K$ for $i \in [n]$, \label{item:stop}
    		\item We have $f_1 \exp(U_1) = \gamma_1 \dots \gamma_{S_1}$ and for $2 \leq i \leq n$ we have $f_i \exp(U_i) = \gamma_{T_{i-1} + 1} \cdots \gamma_{S_i}$. Furthermore for each $i$ we have that $f_i$ is $\mathscr{A}_i$-measurable, \label{item:f_idef}
    		\item $h_i = \gamma_{S_i + 1} \cdots \gamma_{T_i}$ and $h_i$ is $\mathscr{A}_i$-measurable, \label{item:h_idef}
    		\item $\rho(f_i) < 1$ for all $1 \leq i \leq n$, \label{item:rhof_i}
    		\item Whenever $|b(h_i)| > A$, we have $U_i = 0$, \label{item:bh_i}
    		\item For each $1 \leq i \leq n$ we have $$|U_i| \leq \rho(f_1 h_1 f_2 h_2 \cdots h_{i-1}f_{i})^{-1} r,$$ \label{item:U_ibound}
    		\item For each $1 \leq i \leq n$,  we have that $U_i$ is conditionally independent of $\mathscr{A}_n$ given $\mathscr{A}_i$, \label{item:U_1ind1}
    		\item The $U_i$ are conditionally independent given $\mathscr{A}_n$, \label{item:U_iind2}
    		\item For each $1 \leq i \leq n$, it holds  $$\E\left[   \frac{\var(\rho(f_i) U(f_i) U_i b(h_i)|\mathscr{A}_i)}{\rho(f_1h_1 f_2h_2 \cdots f_{i-1}h_{i-1})^{-2} r^2} \,|\, \mathscr{A}_{i-1}  \right] \geq m_i I.$$ \label{item:m_i}
    	\end{enumerate}
    \end{definition}
    
    Note that in \ref{item:m_i} by $\var$ we mean the covariance matrix and we are using the ordering given by positive semi-definiteness \eqref{MatrixPartialOrder} and we denote as in \eqref{NotationConvention} by $U_ib(h_i) = \psi_{b(h_i)}(U_i)$.
    
    A proper decomposition as above gives us 
    \begin{equation}
    \gamma_1\cdots \gamma_{T_n} = f_1\exp(U_1)h_1f_2 \exp(U_2)h_2 \cdots h_{n-1}f_n\exp(U_n)h_n
    \end{equation}
    
    As in \cite{KittleKogler}, we briefly comment on the various properties of proper decompositions for the convenience of the reader. We use the parameter $K$ and \ref{item:stop} to ensure that each of the $f_ix$ and $h_ix$ for $x \in \R^d$ are close in distribution to $\nu$. Properties \ref{item:rhof_i}, \ref{item:bh_i} and \ref{item:U_ibound} are needed in order to apply Proposition~\ref{MainTaylorBound}. We require \ref{item:U_1ind1} so that we have $\Var(U_i|\mathscr{A}_n) = \mathrm{Var}(U_i|\mathscr{A}_i)$ and in particular the latter is a $\mathscr{A}_i$-measurable random variable. \ref{item:U_iind2} is needed so that $U_1, \dots, U_n$ are conditionally independent given $\mathscr{A}$ and therefore we can apply Berry Essen type results to approximate the disintegration of the measure as a normal distribution.
    
    One works with two sequences of random variables $f$ and $h$ instead of one in order to be able to concatenate proper decompositions as in Proposition~\ref{VarSumAdds}. Indeed, if we had proper decompositions of the form $$\gamma_1\cdots \gamma_{T_n} = g_1\exp(U_1)g_2 \exp(U_2)g_3 \cdots g_{n}f\exp(U_n)g_{n + 1}$$ we could show a variant of \eqref{BasicConcat} and all other results on proper decompositions. However we could not prove anything like Proposition~\ref{VarSumAdds}, whose flexible choice of the parameter $M$ is useful in combining information from various scales. 
    
    We next define the $V$ function mentioned in the outline of proofs. The additional parameter $\kappa > 0$ is introduced in order to be able to concatenate the decompositions in a suitable way (Proposition~\ref{VarSumAdds}).

    \begin{definition}\label{VDef}
    	Given $(\mu,n,K,A)$ and $\kappa,r > 0$ we denote by $$V(\mu,n,K, \kappa ,A; r)$$ the \textbf{variance sum} defined as the supremum for $k = 0,1,2, \ldots , n$ of all possible values of $$\sum_{i = 1}^k m_i$$ for a proper decomposition of $(\mu,k,K,A)$ at scale $r$ with $\rho(f_1h_1 \cdots f_k h_k) \geq \kappa$ almost surely. 
    \end{definition}
    
    It is clear that for any $\kappa' > 0$ with $\kappa' \leq \kappa$ we have 
    \begin{equation}\label{TrivialVarianceBound}
    V(\mu,n,K, \kappa' ,A; r) \geq V(\mu,n,K, \kappa ,A; r).
    \end{equation}
    
    \subsection{Preliminary results}\label{section:ExistenceProper}
    
    In this section we recall some results from \cite{KittleKogler} that are used to show that the variance sum is large. First, we  can construct proper decompositions comparing the variance and the trace.
    
    \begin{proposition} \label{InitialDecompositionV}
    	Let $\mu$ be a finitely supported, contracting on average and irreducible  probability measure on $G$ and assume that $\{ \rho(g) \,:\, g \in \mathrm{supp}(\mu) \} \subset [R^{-1},R]$ for some $R > 1$. Then there exists constants $A = A(\mu)$ and $c_1 = c_1(\mu)$ such that the following holds.
    	
    	Let $\kappa, s > 0$ be sufficiently small in terms of $\mu$ and let $K$ be sufficiently large in terms of $\mu$. Then $$V(\mu,1,K,R^{-3K} \kappa, A ; R^{-K} \kappa s) \geq c_1 \mathrm{tr}(q_{\tau_{\kappa}}; s). $$
    \end{proposition}
    
    \begin{proof}
    	This follows from Proposition 8.3 of \cite{KittleKogler}. Indeed since $\mu$ is irreducible it is $(c, T)$-well-mixing and $(\alpha_0, \theta, A)$-non-degenerate for suitable $c,T,\alpha_0, \theta$ and $A$.
    \end{proof}

    We next state a result on concatenating decompositions. We note that it is straightforward to show that for any measure $\mu$ and any admissible choice of coefficients, the variance sum is additive
    \begin{align}
    \MoveEqLeft V(\mu, n_1 + n_2, K, \kappa_1\kappa_2,A; r) \nonumber \\&\geq V(\mu, n_1, K, \kappa_1,A; r) + V(\mu, n_2, K, \kappa_2,A; \kappa_1^{-1}r). \label{BasicConcat}
    \end{align} 
    
    However, in order to use our trace results on various scales it is necessary to work with different scales $r_1$ and $r_2$ and therefore we show the following proposition.
    
    \begin{proposition}(Proposition 8.4 of \cite{KittleKogler}) \label{VarSumAdds}
    	Let $\mu$ be a probability measure on $G$ and assume that $\{ \rho(g) \,:\, g \in \mathrm{supp}(\mu) \} \subset [R^{-1},R]$ for some $R > 1$. Let $n_1, n_2, K \in \mathbb{Z}_{\geq 0}$ with $n_2, K > 0$ and let $\kappa_1, \kappa_2, r \in (0, 1)$. Let $A > 0$ and let $M \geq R$. Then
    	\begin{align*}
    	\MoveEqLeft V(\mu,  n_1 + n_2, K, R^{-1} M^{-1} \kappa_1 \kappa_2, A; r) \\&\geq V(\mu, n_1, K, \kappa_1, A; r) + V(\mu,  n_2, K, \kappa_2, A; M \kappa_1^{-1} r).
    	\end{align*} 
    \end{proposition}
    
    We also have the following corollaries
    
    \begin{corollary} Let $\mu$ be a probability measure on $G$ and assume that $\{ \rho(g) \,:\, g \in \mathrm{supp}(\mu) \} \subset [R^{-1},R]$ for some $R > 1$. Then for $n,K \in \Z_{\geq 0}$, $\kappa, r \in (0,1)$  and $M \geq R$, \label{corollary:var_sum_adds}
    	\begin{equation*}
    	V(\mu,  n, K, R^{-1}M^{-1} \kappa, A; M^{-1} r) \geq  V(\mu,  n, K, \kappa, A; r)
    	\end{equation*}
    \end{corollary}
    
    \begin{proof}
    	This follows from Proposition \ref{VarSumAdds} by letting $n_1 = 0$ and $\kappa_1 = 1$.
    \end{proof}
    
    \begin{corollary} \label{corollary:inductive_var_sum_adds}
    	Let $\ell \in \Z_{>0}$ and suppose that for $i = 1, \dots, \ell$ we have
    	\begin{equation*}
    	V(\mu, n_i, K, \kappa_i, A ; r_i) \geq V_i.
    	\end{equation*} for tuples $(n_i)_{i = 1}^{\ell}$, $(\kappa_i)_{i = 1}^{\ell}$, $(r_i)_{i = 1}^{\ell}$ and $(V_i)_{i = 1}^{\ell}$.
    	Suppose further that $r_{i+1} \geq R^2 r_i \kappa_i^{-1}$ for $1 \leq i \leq \ell - 1$. Then
    	\begin{equation*}
    	V\left(\mu, \sum_{i=1}^{\ell} n_i, K, R^{-2} \kappa_{\ell} r_{\ell}^{-1} r_1, A ; R^{-1}r_1\right) \geq \sum_{i=1}^{\ell} V_i.
    	\end{equation*}
    \end{corollary}
    
    \begin{proof}
    	This follows by induction on $\ell$. The base case follows from Corollary \ref{corollary:var_sum_adds} by setting $M = R$. For the inductive step we simply set $M = R^{-1}\kappa_{\ell} r_{\ell}^{-1} r_{\ell+1}$, which satisfies $M \geq R$ by our assumptions, and apply Proposition~\ref{VarSumAdds}.
    \end{proof}

    \subsection{Variance to infinity}
    
    Here we introduce some results on showing that the variance sum can be large in some situations by looking at entropy gaps.
    
    \begin{proposition}\label{proposition:variance_bound_first} Under the assumptions and with the notation of Proposition~\ref{proposition:entropy_gap_first} and Proposition~\ref{proposition:trace_bound_first}, there exists a constant $\alpha = \alpha(\mu) > 0$ such that there exist sequences $(\tilde{\kappa}_m)_{m=1}^{\infty}$ and $(\tilde{r}_m)_{m=1}^{\infty}$ such that for all sufficiently large $m$ we have $\tilde{r}_m^{-1} \tilde{\kappa}_m \geq \exp(\exp(\sqrt{m}))$ and
    	$$V(\mu, m, K, \tilde{\kappa}_m, A; \tilde{r}_m) \geq m\alpha$$
    	where $K = \lceil\exp(\sqrt{m}) \rceil$.
    \end{proposition}
    
    \begin{proof}
    	Let $m \in \Z_{\geq 1}$ and write $K = \lceil \exp(\sqrt{m}) \rceil$. Let $(\kappa_n)_{n=1}^{\infty}$ and $(r_n)_{n=1}^{\infty}$ be as in Proposition~\ref{proposition:trace_bound_first} with $r_n \in (\kappa_n^{\beta}, \kappa_n^{\delta})$. By passing to a subsequence we may assume without loss of generality that for all $n$
    	\begin{equation} \label{kappacondition}
    	r_n \geq R^{3K} r_{n - 1}.
    	\end{equation}
    	
    	Choose $N$ large enough with $\log \log r_N^{-1} > m$. Note that by Proposition~\ref{InitialDecompositionV} there is $\alpha_1 = \alpha_1(\mu) > 0$ such that $$V(\mu,1,K,R^{-3K} \kappa_{i}, A; R^{-K}\kappa_i r_i) \geq \alpha_1$$
    	for all $i \geq N$. Applying Corollary \ref{corollary:inductive_var_sum_adds} with $m$ in the role of $\ell$, $R^{-K} \kappa_{N+m-1} r_{N+m-1}$ in the role of $r_1$ and so on with $R^{-K} \kappa_{N} r_{N}$ in the role of $r_{\ell}$ as well as $R^{-3K} \kappa_{N+m-1}$ in the role of $\kappa_1$ and so on and $R^{-3K} \kappa_{N}$ in the role of $\kappa_{\ell}$ we note that by \eqref{kappacondition} we have
    	\begin{equation*}
    	V(\mu, m, K, R^{-3K - 2} \kappa_{N+m-1} r_{N+m-1} r_N^{-1}, A ; R^{-K - 1} \kappa_{N+m-1} r_{N+m-1}) \geq m \alpha_1.
    	\end{equation*}
    	
    	We now let $\tilde{\kappa}_m = R^{-3K - 2}\kappa_{N+m-1} r_{N+m-1} r_N^{-1}$ and $\tilde{r}_m = R^{-K - 1} \kappa_{N+m-1} r_{N+m-1}$. Note that
    	\begin{align*}
    	\tilde{r}_m^{-1} \tilde{\kappa}_m &= R^{-2K - 1} r_N^{-1}.
    	\end{align*}
    	Note that $R^{2K + 1}$ is double exponential in $\sqrt{m}$ whereas $r_N^{-1}$ is double exponential in $m$ and therefore $\tilde{r}_m^{-1} \tilde{\kappa}_m \geq \exp(\exp(\sqrt{m}))$ for all sufficiently large $m$.
    \end{proof}

    \begin{proposition} \label{proposition:use_many_gaps}
    	Suppose that $\mu$ is a contracting on average, finitely supported, irreducible probability measure on $G$ and that $\dim \nu < \frac{h_{\mu}}{|\chi_{\mu}|}$. Suppose further that there is some $B > 0$ such that for all sufficiently large $n$ we have $$\log M_n \geq - n \exp( (\log n)^{B}).$$
    	Then there is some $\alpha > 0$ such that as $\kappa \to 0$ we have
    	\begin{equation*}
    	V(\mu, m, K, \kappa^{ \exp((\log \log \kappa^{-1})^{B + \eps_1}) + o_{\mu}(1)}, A; \kappa^{ \delta + \exp((\log \log \kappa^{-1})^{B + \eps_1}} )) \gg_{\mu} \frac{\alpha}{(\log \log \kappa^{-1})^{B + 2\eps_1}}
    	\end{equation*}
    	where $m \ll \exp((\log \log \kappa^{-1} )^ {B} )$ and $K \leq \exp(\sqrt{\log \log \kappa^{-1}})$ sufficiently large in terms of $\mu$.
    \end{proposition}

    \begin{proof}
    	First note that by Proposition \ref{proposition:trace_sum_second} for every $\eps_1 > 0$ the following holds for $\kappa$ sufficiently small: There exist $s_1, \ldots, s_{m} > 0$ with $m$ and integer satisfying $m \asymp \exp((\log \log \kappa^{-1})^{B + \eps_1}))$ such that for each $1\leq i \leq m$, we have that
    	$$s_i \in (\kappa^{\exp((\log \log \kappa^{-1})^{B + \eps_1})}, 2(\log \log \kappa^{-1})^{\frac{B}{2} + \eps_1} \kappa^{\delta}),$$
    	for each $1 \leq i \leq m-1$ we have $s_{i + 1} \geq \kappa^{-3}s_i$ and  it holds that $$\sum_{i = 1}^m \mathrm{tr}(q_{\tau_{\kappa}}; s_i) \geq \frac{\alpha}{(\log \log \kappa^{-1})^{B + 2\eps_1}}.$$
    	By Proposition \ref{InitialDecompositionV} we have
    	$$V(\mu,1,K,R^{-3K} \kappa, A ; R^{-K} \kappa s_i) \gg  \mathrm{tr}(q_{\tau_{\kappa}}; s_i) $$
    	and so by Corollary \ref{corollary:inductive_var_sum_adds} with setting $\kappa_i = R^{-3K}\kappa$ and $r_i = R^{-K} \kappa s_i$ we get
    	$$V(\mu,m,K,R^{-3K - 1} \kappa s_{m}^{-1} s_1, A ; R^{-K-1} \kappa s_1) \gg_{\mu} \frac{1}{(\log \log \kappa^{-1})^{B + 2\eps_1}}.$$
    	Applying Corollary \ref{corollary:var_sum_adds} with $M = s_m^{-1}$ and \eqref{TrivialVarianceBound} it follows that $$V(\mu,m,K,R^{-4K} \kappa s_1, A ; R^{-K-1} \kappa s_m s_1) \gg_{\mu} \frac{1}{(\log \log \kappa^{-1})^{B + 2\eps_1}}.$$ The result now follows from Corollary \ref{corollary:var_sum_adds} and \eqref{TrivialVarianceBound} after replacing $\delta$ by a slightly smaller constant.
    \end{proof}

    \section{Proof of Theorem~\ref{GeneralHochman} and Theorem~\ref{MainWeakHochman}}
    \label{SectionConclusion}

    The main result of this section will be to prove Theorems \ref{GeneralHochman} and \ref{MainWeakHochman}. The last remaining ingredient we need for this is the following proposition.
    
    \begin{proposition} \label{proposition:sufficient_for_full_dimension}
    	Let $\mu$ be a contracting on average probability measure on $G$ with self-similar measure $\nu$. Let $n_i,  \kappa_i, K_i, r_i$ be sequences and let $V_i = V(\mu, n_i, \kappa_i, K_i, A; r_i)$. Suppose that as $i \to \infty$ it holds that
    	\begin{equation}
    	V_i / \log n_i \to \infty, \quad \quad \quad \quad  K_i / \log n_i \to \infty\label{eq:v_i_k_i_diverge}
    	\end{equation}
    	and
    	\begin{equation}
    	\frac{\log r_i^{-1} - \log \kappa_i^{-1}}{n_i} \to \infty. \label{eq:taylor_term_dies}
    	\end{equation}
    	Then $\dim \nu = d.$
    \end{proposition}
    
    This is a corollary of Proposition~\ref{propositon:normal_to_entropy} and of the following lemma, which relies on a Berry-Essen-type theorem.
    
    \begin{lemma} \label{lemma:decomposition_to_normal}
    	Let $\mu$ be a contracting on average probability measure on $G$ and let $A > 0$. Then there are constants $c, C = c(\mu, A), C(\mu, A) > 0$ such that the following holds for all sufficiently small $r > 0$ (in terms on $\mu$ and $A$). Let $V = V(\mu, n, \kappa, K, A ; r)$ and let $x$ be a sample from $\nu$. Suppose that $V$ is sufficiently large. Then we can find some $\sigma$-algebra $\hat{\mathscr{A}}$ such that there is some $\hat{\mathscr{A}}$-measurable random positive semi definite symmetric matrix $\Sigma$ with $$\mathbb{P}\left[\Sigma \geq \frac{Vr^2}{4}I\right] \geq 1 -  O\left( n \exp \left( -c \min \{K, V \} \right) \right) $$ and some $\hat{\mathscr{A}}$-measurable random $x_0 \in \R^d$ such that $$\mathbb{P} [\mathcal{W}_1(x | \hat{\mathscr{A}}, N(x_0, \Sigma)) > C^n \kappa^{-1} r^2 + C r] \geq 1 - O\left( n \exp \left( -c K \right) \right).$$
    \end{lemma}

    \begin{proof}[Proof of Proposition \ref{proposition:sufficient_for_full_dimension}]
    	First note that by Proposition \ref{VarSumAdds} for any $r < R^{-1} r_i$ we have $V(\mu, n_i, K_i, R^{-1} r \kappa_i / r_i, A; r) \geq V_i$. By Lemma \ref{lemma:decomposition_to_normal} we can construct some sample $x$ from $\nu$, some $\sigma$-algebra $\hat{\mathscr{A}}$, some $\hat{\mathscr{A}}$-measurable random positive semi-definite $\Sigma$ and some $\hat{\mathscr{A}}$-measurable random $x_0 \in \R^d$ such that
    	\begin{align*}
    	\MoveEqLeft \mathbb{P} \left[ \Sigma \geq V_i r^2 I / 4 \text{ and } \mathcal{W}_1(x | \hat{\mathscr{A}}, N(x_0, \Sigma)) \leq R C^{n_i} \kappa_i^{-1} r_i r + C r \right] \\ &\geq 1 - O\left( n_i \exp \left( -c \min \{K_i, V_i \} \right) \right)
    	\end{align*}
    	for some constants $c, C >0$ which depend only on $\mu$ and $A$. By applying Proposition \ref{propositon:normal_to_entropy} with $C^{n_i} \kappa_i^{-1} r_i r + C r $ in the role of $r$ we see that it is sufficient to show that
    	\begin{equation}
    	\frac{\sqrt{V_i}}{C^{n_i} \kappa_i^{-1} r_i + C} \to \infty \label{eq:ratio}
    	\end{equation}
    	and
    	\begin{equation}
    	n_i \exp \left( -c \min \{K_i, V_i \} \right) \to 0 \label{eq:prob_to_zero}.
    	\end{equation}
    	First note that \eqref{eq:prob_to_zero} follows from \eqref{eq:v_i_k_i_diverge}. For \eqref{eq:ratio} first note that since $V_i \to \infty$ it is sufficient to show that $C^{n_i} \kappa_i^{-1} r_i \to 0$. This follows from \eqref{eq:taylor_term_dies}.
    \end{proof}

    It remains to prove Lemma~\ref{lemma:decomposition_to_normal}. We will need the following quantitative version of Cramer's theorem.
    
    \begin{lemma}\label{lemma:cramer} 
    	There is some absolute constant $c>0$ such that the following is true. Suppose that $X_1, \dots, X_n$ are random $d \times d$ symmetric positive semi-definite matrices such that $X_i \leq b I$ for some $b > 0$ and
    	$$\mathbb{E}[X_i | X_1, \dots, X_{i-1}] \geq m_i I.$$
    	Suppose that $\sum_{i=1}^{n}m_i = an$. Then $$\log \mathbb{P}\left[X_1 + \dots + X_n \leq \frac{na}{4}I\right] \leq -can + O(d \log (b / a))$$
    \end{lemma} 
    
    Lemma \ref{lemma:cramer} is a corollary of the following result.
    
    \begin{corollary}(Corollary 7.9 of \cite{Kittle2023})\label{EffectiveCramerR}
    	There is a constant $c > 0$ such that the following is true for all $a \in [0,1)$ and $n\geq 1$. Let $X_1, \ldots , X_n$ be random variables taking values in $[0,1]$ and  let $m_1, \ldots , m_n \geq 0$ be such that we have almost surely
    	$\E[X_i | X_1 , \ldots , X_{i-1}] \geq m_i$ for $1 \leq i \leq n$. Suppose that $\sum_{i = 1}^n m_i = an$. Then $$\log \mathbb{P}\left[X_1 + \ldots + X_n \leq \frac{1}{2}na\right] \leq - cna.$$
    \end{corollary}
    
    \begin{proof}[Proof of Lemma \ref{lemma:cramer}]
    	For convenience write $Y_n = X_1 + \ldots + X_n$ and choose a set $S$ of unit vectors in $\R^d$ such that if $y$ is any unit vector in $\R^d$ then there exists $x \in S$ with $\| x-y \| \leq \frac{a}{8b}$. Note that we may choose $S$ such that $|S| \leq O((b / a) ^{d-1})$.
    	
    	By Corollary~\ref{EffectiveCramerR} we know that for any $x \in S$, $$\log \mathbb{P}\left[x^{T}Y_n x \leq \frac{na}{2}\right] \leq - can.$$ Let $A$ be the event that there exists some $x \in S$ with $x^{T}Y_n x \leq \frac{na}{2}$. We have that $\log \mathbb{P}[A]$ is at most $-can + \log |S|$. It suffices therefore to show that on $A^C$ we have $Y_n \geq \frac{na}{4}I$.
    	
    	Indeed let $y \in \R^d$ be a unit vector. Choose some $x \in \R^d$ with $\| x-y\| \leq a /8 b$. Suppose that $A^C$ occurs. Note that we must have $Y_n \leq bnI$ and therefore $||Y_n|| \leq bn$. This means
    	\begin{align*}
    	y^TY_n y & = xY_n x + x^TY_n (y-x) + (y-x)^TY_n y\\
    	& > \frac{an}{2} - 2bn \cdot \frac{a}{8b}  = \frac{an}{4}.
    	\end{align*}
    	and result follows.
    \end{proof}
    
    We now prove Lemma \ref{lemma:decomposition_to_normal}, which is similar to Proposition 8.6 from \cite{KittleKogler}, yet simpler as we only need to apply some straightforward consequences from the Berry-Essen theorem. Indeed, the main engine is the following lemma from \cite{KittleKogler}.
    
    \begin{lemma}(Lemma 5.9 of \cite{KittleKogler})\label{BerryEssenType}
    	Let $X_1, X_2, \ldots , X_n$ be independent random variables taking values in $\R^d$ and denote for each $i \in [n]$ write $$\Sigma_i = \mathrm{Var} \, X_i.$$ Suppose that $\delta > 0$ is such that for each $i \in [n]$ we have $|X_i| \leq \delta$ almost surely. Let $\Sigma = \sum_{i = 1}^n \Sigma_i$ and $S = X_1 + \ldots + X_n$. Then $$\mathcal{W}_1(S, N(\mathbb{E}[S], \Sigma)) \ll_d \delta.$$ 
    \end{lemma} 
    
    \begin{proof}[Proof of Lemma \ref{lemma:decomposition_to_normal}]
    	The proof further relies on some basic lemmas from \cite{KittleKogler}. Suppose that $(f, h, U, \mathscr{A}, \gamma, \mathscr{F}, S, T, m)$ is a proper decomposition of $(\mu, n, K, A)$ at scale $r$ such that $\sum_{i=1}^n m_i \geq V / 2$ and let $v$ be an independent sample from $\nu$. Let $$I = \{ i \in [1, n] \cap \Z : |b(h_i)| \leq A \}$$ and let $m = |I|$. Enumerate $I$ as $i_1 < i_2 < \cdots < i_m$ and define $g_1, \dots, g_m$ by $g_1 = f_1 h_1 \dots f_{i_1}$ and $g_j = h_{i_{j-1}} f_{i_{j-1}+1} \dots f_{i_j}$ for $2 \leq j \leq m$. Define $\overline{v}$ by $\overline{v} = h_{i_m} f_{i_m +1} \dots h_n v$ and let $V_j = U_{i_j}$. Let $x$ be defined by $$x = g_1 \exp(V_1) \dots g_m \exp(V_m) \overline{v}.$$ Note that $x$ is a sample from $\nu$. Let $\hat{\mathscr{A}}$ be the $\sigma$-algebra generated by $\mathscr{A}_n$ and $v$. Note that the $g_j$ and $\overline{v}$ are $\hat{\mathscr{A}}$-measurable.
    	
    	We prove the proposition by showing that with high probability we can apply Proposition~\ref{MainTaylorBound} to $g_1, \dots, g_m$, $V_1, \dots, V_m$, and $\overline{v}$.
    	
    	Let $E$ be the event that $|\overline{v}| \leq 2A$ and that for each $j = 1, \dots, m$ we have $|b(g_j)| \leq 2A$, $\rho(g_j) < 1$ and $|V_j| \leq \rho(g_1 \dots g_j)^{-1} r$. By \cite{KittleKogler}*{Corollary 3.12} we know that $\mathbb{P}[E^C] \leq \exp(-c_1 K)$ for some $c_1 = c_1(\mu, A) > 0$. 
    	
    	For $j=1, \dots, m$ define $\zeta_j$ by $$\zeta_j = D_u(g_1 \cdots g_j \exp(u) g_{j + 1} \cdots g_m \overline{v})|_{u=0}.$$ By Proposition~\ref{MainTaylorBound} on $E$ we have $$\left|x-g_1 \dots g_m \overline{v} - \sum_{j=1}^m \zeta_j(V_j)\right| \leq C_1^m \rho(g_1 \dots g_m)^{-1} r^2$$ for some $C_1 = C_1(A) > 0$. Clearly the right hand side is at most $C_1^n \kappa^{-1} r^2$.
    	
    	Let $F$ be the event that $$\sum_{j=1}^m \var \zeta_j(V_j|\hat{\mathscr{A}}) \geq \frac{Vr^2}{4}I.$$ Setting $$x_0 = x-g_1 \dots g_m \overline{v} - \E\left[\sum_{j=1}^m \zeta_j(V_j|\hat{\mathscr{A}})\right] \quad \text{ and } \quad \Sigma = \sum_{j=1}^m \var \zeta_j(V_j|\hat{\mathscr{A}}),$$ the claim follows from Lemma~\ref{BerryEssenType} provided we can bound $\mathbb{P}[F]$. 
    	
    	For $i=1, \dots, n$ define $$\hat{\zeta}_i = D_u(f_1 h_1 \cdots h_{i-1}f_i \exp(u) b(h_i))|_{u=0}$$ and let $\underline{F}$ be the event that
    	\begin{equation*}
    	\left\| \sum_{i=1}^{n} \var \hat{\zeta}_i(U_i|\hat{\mathscr{A}}) - \sum_{j=1}^m \var \zeta_{i_j}(V_j|\hat{\mathscr{A}}) \right\| < r^2.
    	\end{equation*}
    	
    	Let $\overline{F}$ be the $\hat{\mathscr{A}}$-measurable event that $\sum_{i=1}^{n} \var(\hat{\zeta}_i(U_i)|\hat{\mathscr{A}}) \geq (V / 4 + 1) r^2 I$. Clearly $\underline{F} \cup \overline{F} \supset F$ so it suffices to bound $\mathbb{P}[\underline{F}^C]$ and $\mathbb{P}[\overline{F}^C]$.
    	
    	Since $g_1, \ldots , g_m$ and $\overline{v}$ are $\hat{\mathscr{A}}$ measurable, by \cite{KittleKogler}*{Lemma 3.3} we have for $j=1, \dots, m$ that $\var (\zeta_j(V_j)|\hat{\mathscr{A}})$ is equal to $$ \rho(g_1 \dots g_j)^2\cdot U(g_1 \dots g_j)\psi_{g_{j + 1} \dots g_m \overline{v}} \circ \var (V_j |\hat{\mathscr{A}}) \circ \psi_{g_{j + 1} \dots g_m}^TU(g_1 \dots g_j)^T$$ and that $$\var (\hat{\zeta}_{i_j}(U_{i_j})|\hat{\mathscr{A}}) = \rho(g_1\cdots g_j)^2 \cdot U(g_1 \dots g_j)\psi_{b(h_{i_j})} \circ \var (V_j|\hat{\mathscr{A}}) \circ \psi_{b(h_{i_j})}^TU(g_1 \dots g_j)^T.$$ We also have that $| V_j | \leq \rho(g_1 \cdots g_j)^{-1} r$ almost surely and so consequently $\| \var V_j \| \leq \rho(g_1\cdots g_j)^{-2} r^2$. Therefore by \cite{KittleKogler}*{Lemma 3.1 (iii)}, $$\| \var \zeta_j(V_i| \hat{\mathscr{A}}) - \var \hat{\zeta}_{i_j}(U_{i_j}| \hat{\mathscr{A}}) \| \ll_d |b(h_j) - g_{j + 1} \dots g_m \overline{v}|^2 r^2.$$ Furthermore we have that whenever $i \notin I$ that $\var (\hat{\zeta}_{i}(U_{i})|\hat{\mathscr{A}}) = 0$. We may assume without loss of generality that $n \exp(-K \chi_{\mu} / 10) < 1$. This means that, providing $K$ is sufficiently large (in terms of $d$), in order for $\underline{F}$ to occur it is sufficient that for each $j = 1, \dots, m$ we have
    	$$|b(h_j) - g_{j + 1} \dots g_m \overline{v}| < \exp(-K \chi_{\mu} / 10) < 1/n.$$
    	By \cite{KittleKogler}*{Corollary 3.12} this occurs with probability at least $1 - m\exp(-c_2 K)$ for some $c_2 = c_2(\mu) > 0$ and therefore $\mathbb{P}[\underline{F}^C] \leq m\exp(-c_2 K) \leq n\exp(-c_2 K)$.
    	
    	Finally we wish to bound $\mathbb{P}[\overline{F}^C]$. Let 
    	\begin{align*}
    	\Sigma_i &= r^{-2} \var(\hat{\zeta_i}(U_i)|\hat{\mathscr{A}}) =  r^{-2}\var(\hat{\zeta_i}(U_i)|\mathscr{A}_i) \\
    	&= r^{-2}\Var(\rho(f_1h_1\cdots h_{i-1}f_i)U(f_1h_1\cdots h_{i-1}f_i)U_ib(h_i)|\mathscr{A}_i)
    	\end{align*}
    	By construction we know that $$\mathbb{E}[\Sigma_i | \Sigma_1, \dots, \Sigma_{i-1}] \geq m_i I.$$
    	We also know that $\| \Sigma_i \| \leq A^2$ since $||\psi_{b(h_i)}|| \leq |b(h_i)| \leq A$. This means that we can apply Lemma \ref{lemma:cramer} to conclude that $\mathbb{P}[\overline{F}^C] \leq \exp \left( - c_2 V  + O(d \log n)\right)$ for some constant $c_2=c_2(A)>0$. Clearly this means that $$\mathbb{P}[\overline{F}^C] \leq n \exp \left( - c_3 V \right)$$ for some constant $c_3 = c_3(d) > 0$. The result follows by combining these estimates.
    \end{proof}
    
    \subsection{Proof of Theorem \ref{GeneralHochman} and Theorem \ref{MainWeakHochman}}

    Now we prove Theorems \ref{GeneralHochman} and \ref{MainWeakHochman}
    
    \begin{proof}[Proof of Theorem \ref{MainWeakHochman}]
    	If $\dim \nu < h_{\mu} / |\chi_{\mu}|$ then Proposition~\ref{proposition:trace_bound_first} and Proposition~\ref{proposition:variance_bound_first} hold and therefore $\dim \nu = d$ by Proposition \ref{proposition:sufficient_for_full_dimension}.
    \end{proof}
    
    From Proposition \ref{proposition:sufficient_for_full_dimension} and our estimates on the entropy gaps we can prove Theorem \ref{MainWeakHochman}.
    
    \begin{proof}[Proof of Theorem \ref{MainWeakHochman}] Assume that $\dim \nu < h_{\mu} / |\chi_{\mu}|$.
    	Let $$\underline{k}_n := \exp(  - \exp ( n^{3/2})).$$ Then by Proposition \ref{proposition:use_many_gaps} with setting $\kappa = k_n$ and $B = \frac{1}{3} - \eps$ it follows for $\underline{m}_n \asymp \exp(n^{1/2 - \eps'})$, $\underline{K}_n = \exp(n^{3/4})$ as well as $$\underline{\kappa}_n = k_n^{\exp(n^{1/2 - 2\eps'}) + o_\mu(1)} \quad\quad \text{ and } \quad\quad \underline{r}_n = k_n^{\delta + \exp(n^{1/2 - 2\eps'})}$$ for sufficiently small $\eps' > 0$ that \begin{equation}
    	V(\mu, \underline{m}_n, \underline{K}_n, \underline{\kappa}_n, A; \underline{r}_n) \geq \underline{V}_n. \label{eq:from_earlier_prop}
    	\end{equation} for $\underline{V}_n = \frac{\alpha}{n^{1/2 - \eps'}}.$
    	
    	We fix some large even $N$ and apply Corollary \ref{corollary:inductive_var_sum_adds} with $\ell = N/2 + 1$, $r_i = \underline{r}_{N+1 - i}$ and $\kappa_i = \underline{\kappa}_{N+1 - i}$. We just need to check that $\underline{\kappa}_{N-i+1} \underline{r}_{N-i} > R^2 \underline{r}_{N-i+1}$. Letting $b = N-i$ providing $N$ is sufficiently large we get
    	
    	\begin{align*}
    	R^{-2} \underline{\kappa}_{N-i+1} \underline{r}_{N-i+1} ^{-1} &= R^{-2} k_{N-i + 1}^{o_{\mu}(1)-\delta} \\ 
    	& > R^{-2} \exp( \delta/2 \exp((b + 1)^{3/2}))) \\
    	& > R^{-2} \exp( \delta/2 \exp ( (b^{ 3/2} + 3 b^{1/2}/2))) \\
    	&> R^{-2} \exp(  2\exp ( b ^{ 3/2} + b^{1/2})) \\
    	& > \exp(  \delta \exp ( b ^{ 3/2}) +\exp ( b ^{ 3/2} + b^{1/2 -2\varepsilon'})) \\
    	& = \underline{r}_{N-i}^{-1}.
    	\end{align*}
    	
    	Therefore by Corollary \ref{corollary:inductive_var_sum_adds} we have
    	\begin{equation*}
    	V(\mu, \sum_{N/2}^N \underline{m}_n, \underline{K}_{N/2}, R^{-2} \underline{r}_N \underline{\kappa}_{N/2} \underline{r}_{N/2} ^{-1}, A ; R^{-1}\underline{r}_N) \geq \sum_{N/2}^N \underline{V}_n.
    	\end{equation*}
    	
    	We now write $V_N = \sum_{n=N/2}^{N} \underline{V}_n \gg_{\mu, \varepsilon} N^{1/2 + \varepsilon'} $, $K_N = \underline{K}_{N/2} \gg_{\mu, \varepsilon} \exp((N/2)^{3/4})$, $n_N = \sum_{n=N/2}^N \underline{m}_n \gg \exp( N^{1/2 - \varepsilon'})$, $\kappa_N = R^{-2} \underline{r}_N \underline{\kappa}_{N/2} \underline{r}_{N/2} ^{-1} \gg \exp(\delta \exp((N/2)^{3/2}))\underline{r}_N$ and $r_N = \underline{r}_N$. 
    	
    	It just remains to check \eqref{eq:v_i_k_i_diverge} and \eqref{eq:taylor_term_dies}. For \eqref{eq:v_i_k_i_diverge} note
    	\begin{equation*}
    	\frac{V_N}{\log n_N} \gg \frac{N^{1/2 + \varepsilon'}}{N^{1/2 - \varepsilon'}} \to \infty \quad\quad \text{ and } \quad\quad \frac{K_N}{\log n_N} \gg \frac{\exp((N/2)^{3/4})}{N^{1/2-\varepsilon'}} \to \infty
    	\end{equation*}
    	and for \eqref{eq:taylor_term_dies} note
    	\begin{equation*}
    	\frac{\log r_N^{-1} - \log \kappa_N^{-1}}{n_N} \gg \frac{\delta \exp((N/2)^{3/2})}{\exp( N^{1/2-\varepsilon'})} \to \infty.
    	\end{equation*} This concludes the proof of Theorem~\ref{MainWeakHochman}.
    \end{proof}

\bibliography{referencesgeneral.bib}
\end{document}